\documentclass[12pt]{amsart}

\usepackage{amssymb,amsthm,amsmath,amsfonts}
\usepackage{graphicx}
\usepackage{subcaption}
\usepackage{epstopdf}
\usepackage{mathtools}

\usepackage{tikz,pgfplots,pgfplotstable}

\usepgfplotslibrary{external}
\tikzsetexternalprefix{pictures/}
\usepackage{enumerate}

\usepackage{commath}
\usepackage{color}

\usepackage[colorlinks,anchorcolor=black,citecolor=black,linkcolor=black]{hyperref}
\usepackage[hypcap]{caption}
\usepackage[a4paper,margin=2cm]{geometry}

\allowdisplaybreaks{}

\DeclareMathOperator{\Res}{\mathrm{Res}}
\DeclareMathOperator{\Ord}{\mathcal{O}}
\let\Im\undefined
\DeclareMathOperator{\Im}{\mathrm {Im}}
\def\R{{\mathbb R}}

\def\C{{\mathbb C}}
\newcommand\IC{{\mathbb C}}

\newcommand\IZ{{\mathbb Z}}
\newcommand\IN{{\mathbb N}}
\newcommand{\nontang}[1]{  \sphericalangle\text{-}\lim_{#1}}

\newcommand{\mgf}{M}
\newcommand{\mgfpsi}{\psi}

\newcommand{\abcgf}[2]{\eta_{#2}^{#1}}

\def\<{\langle}
\def\>{\rangle}

\def\E{{\mathbb E}}

\def\i{\underline i}

\DeclareMathOperator{\NC}{\mathit{NC}} \DeclareMathOperator{\NCirr}{\NC^\mathit{irr}} 
\DeclareMathOperator{\IP}{\mathit{Int}} \newcommand{\alg}[1]{\mathcal{#1}}
\newcommand{\bigabs}[1]{\bigl\lvert#1\bigr\rvert}

\newcommand{\biggabs}[1]{\biggl\lvert#1\biggr\rvert}

\newcommand{\harpleftsign}{\scriptstyle\leftharpoonup}
\newcommand{\harpleft}[2]{  \ifx\displaystyle#1\doalign{$\harpleftsign$}{#1#2}\fi
  \ifx\textstyle#1\doalign{$\harpleftsign$}{#1#2}\fi
  \ifx\scriptstyle#1\doalign{\scalebox{.6}[.9]{$\harpleftsign$}}{#1#2}\fi
  \ifx\scriptscriptstyle#1\doalign{\scalebox{.5}[.8]{$\harpleftsign$}}{#1#2}\fi
}

\newcommand{\harprightsign}{\scriptstyle\rightharpoonup}
\newcommand{\harpright}[2]{  \ifx\displaystyle#1\doalign{$\harprightsign$}{#1#2}\fi
  \ifx\textstyle#1\doalign{$\harprightsign$}{#1#2}\fi
  \ifx\scriptstyle#1\doalign{\scalebox{.6}[.9]{$\harprightsign$}}{#1#2}\fi
  \ifx\scriptscriptstyle#1\doalign{\scalebox{.5}[.8]{$\harprightsign$}}{#1#2}\fi
}
\newcommand{\doalign}[2]{ {\vbox{\offinterlineskip\ialign{\hfil##\hfil\cr#1\cr$#2$\cr}}}}

\usepackage{yhmath,mathabx}

\newcommand{\be}{\begin{equation}}
\newcommand{\ee}{\end{equation}}
\newtheorem{theorem}{Theorem}[section]
\newtheorem{proposition}[theorem]{Proposition}
\newtheorem{corollary}[theorem]{Corollary}
\newtheorem{lemma}[theorem]{Lemma}
       
      \newtheorem{definition}[theorem]{Definition}

\theoremstyle{remark}
\newtheorem{remark}[theorem]{Remark}

\newtheorem{example}[theorem]{Example}

\title{Boolean Cumulants and Subordination in Free Probability}

\author[F. Lehner]{Franz Lehner}
\address[F. Lehner]{Institute of Discrete Mathematics\\
Graz University of Technology\\
Steyrergasse 30\\
A-8010 Graz, Austria}
\email{lehner@math.tu-graz.ac.at}

\author[K. Szpojankowski]{Kamil Szpojankowski}
\address[K. Szpojankowski]{Wydzia\l{} Matematyki i Nauk Informacyjnych\\
Politechnika Warszawska\\
ul. Koszykowa 75\\
00-662 Warszawa, Poland}
\email{k.szpojankowski@mini.pw.edu.pl}

\thanks{Supported by the Austrian Federal Ministry of Education, Science and
  Research and the Polish Ministry of Science and Higher Education, grants
  N$^{\textrm{os}}$ PL 08/2016 and PL 06/2018\\
KSz: research partially supported by NCN grant 2016/23/D/ST1/01077}

\numberwithin{equation}{section}

\begin{document}
\date{
29.04.2020}

\begin{abstract}
  Subordination is the basis of the analytic approach to free additive and multiplicative convolution. 
  We extend this approach to a more general setting and prove that the conditional expectation $\E_\varphi\left[ (z-X-f(X)Yf^*(X))^{-1}| X\right]$ for
  free random variables $X,Y$ and a Borel function $f$  
  is a resolvent again. 
  This result allows the explicit calculation of the distribution of
  noncommutative polynomials of the form $X+f(X)Yf^*(X)$. 
  The main tool is a new combinatorial formula for conditional expectations 
  in terms of Boolean cumulants and a corresponding analytic formula
  for conditional expectations of resolvents,
  generalizing subordination formulas for both additive and multiplicative
  free convolutions.
  In the final section we illustrate the results with step by step explicit
  computations and an exposition of all necessary ingredients.
\end{abstract}


\maketitle

\tableofcontents{}
\section{Introduction}

Free probability was introduced in \cite{VoiculescuAdd} and it can be
understood as a non-commutative counterpart of classical probability theory. It
is non-commutative in the sense that multiplication of random variables is not
necessarily commutative. The classical notion of stochastic independence is not
very natural in this context, but fortunately there are other notions of
independence which lead to interesting theories and applications.
Among several possible notions of non-commutative independence free independence is the
most prominent one, see
Section~\ref{sec:preliminaries} below. Free independence
shares many  properties with classical independence. In particular,
the joint distribution of free random variables $X$  and $Y$ is
uniquely determined by the individual distributions of $X$ and $Y$.
We can thus define additive and multiplicative convolutions of probability
measures
as follows:
Let  $\mu$ and $\nu$ be probability measures on the real line and
$X$ and $Y$ selfadjoint free random variables with distribution $\mu$ 
and $\nu$. Then the free convolution, denoted $\mu\boxplus\nu$, is defined
as the distribution of  $X+Y$.
Similarly, under the additional assumption that the supports of $\mu$ and $\nu$ are contained in the positive half-line,  one can define free multiplicative convolution of
$\mu$ and $\nu$, denoted by $\mu\boxtimes\nu$, as the distribution of
$X^{1/2}YX^{1/2}$. 
These free convolutions can be studied  by methods of free harmonic analysis,
that is, in terms  of analytic functions well known from Nevanlinna theory.
For a probability measure $\mu$ on real line one defines the so called Cauchy transform via the formula
\begin{align*}
  G_\mu(z)=\int_{\R}\frac{1}{z-x}d\mu(x), \qquad z\in \C^+,
\end{align*}
where $\C^+$ is the complex upper half-plane $\C^+=\{z\in \C; \Im(z)>0\}$. The
Cauchy transform is an analytic map which takes values in $\C^{-}=-\C^{+}$
and has nontangential limit
\begin{equation*}
  \nontang{z\to\infty} zG(z)=1
  .
\end{equation*}
By this we mean
\begin{equation*}
  \lim_{\substack{\abs{z}\to\infty\\
      z\in\Gamma_\alpha}} zG(z)=1
\end{equation*}
for any  $\alpha>0$ where by $\Gamma_\alpha$ we denote the nontangential sector
\begin{equation*}
  \Gamma_\alpha = \{z=x+iy\mid y>0, \abs{x}<\alpha y\}
  .
\end{equation*}
In fact this property characterizes Cauchy transforms,
see \cite[Proposition~5.1]{BerVoicUnboundConv}.

There are two approaches to compute the additive free convolution.
The original method of Voiculescu uses the compositional inverse of
the Cauchy transform.
Consider compactly supported probability measure
$\mu$ then one can define in some neighbourhood of $0$ so called $R$-transform
as $R_\mu(z)=G^{-1}_\mu(z)-1/z$, which has the remarkable property that for
compactly supported probability measures $\mu,\nu$ one has 
\begin{align*}
	R_{\mu\boxplus\nu}(z)=R_\mu(z)+R_\nu(z),
\end{align*}
holding on a common domain of the three functions. 
Thus the $R$-transform plays the role of the logarithm of the characteristic function
in classical probability.
The second method is more popular today and uses subordination
which was discovered in \cite{VoiculescuAnaloguesOfEntropyI}
(and appears implicitly in \cite[Proposition~4]{Woess:1986:nearest}).
Subordination asserts that there exist analytic maps $\omega_1,\omega_2$ such that for $i=1,2$ one has $\omega_i:\C^{+}\mapsto\C^{+},$ $\Im\left(\omega_i(z)\right)\geq\Im(z)$, $\omega_i(iy)/iy\to 0$, when $y\to+\infty$ and
\begin{align}\label{eq:addSubordination}
	G_{\mu\boxplus\nu}(z)=G_\mu\left(\omega_1(z)\right)=G_\nu\left(\omega_2(z)\right), \mbox{ for } z\in\C^{+}.
\end{align}
This observation was generalized in \cite{Biane98} to the framework of
conditional expectations in von Neumann algebras (for details see
Section~\ref{sec:preliminaries} below),
and asserts that for free random variables $X,Y$ one has
\begin{equation}
\label{eq:addCondExp}
\E_\varphi\left[(z-X-Y)^{-1}| X\right]=(\omega_1(z)-X)^{-1},
\end{equation}
where $\omega_1$ is as above. After applying the trace $\varphi$ to both sides of the above equation one immediately gets \eqref{eq:addSubordination}. 

To tackle multiplicative convolution one defines the moment generating functions
\begin{align*}
  \mgfpsi_\mu(z) &= \int \frac{tz}{1-tz}\,d\mu(t)
  \\
  \mgf_\mu(z) &= \int \frac{1}{1-tz}\,d\mu(t) = 1 +\mgfpsi_\mu(z)
  .
\end{align*}
They are related to the Cauchy transform via the identity
$$
G_\mu(z) = \frac{1}{z}\mgf_\mu(1/z)
.
$$
Then multiplicative convolution can be calculated using the $S$-transform
 which is defined as the solution $S_\mu(z)$
of the equation
\begin{equation*}
  \psi_\mu
  \left(
    \frac{z}{1+z}S_\mu(z)
  \right)
  =z
\end{equation*}
in some domain and it was proved by \cite{Voiculescu:1987:multiplication} that
\begin{equation*}
  S_{\mu\boxtimes\nu}(z)=S_\mu(z)S_\nu(z)
  .
\end{equation*}
Alternatively, it was shown in \cite{BelinschiBercoviciPartiallyDefined} that the function
\begin{equation*}
  \Sigma_\mu(z) = \frac{\eta_\mu^{-1}(z)}{z}
,
\end{equation*}
where 
\begin{equation*}
  \eta_\mu(z) = \frac{\mgfpsi_\mu(z)}{1+\mgfpsi_\mu(z)}
\end{equation*}
is the \emph{Boolean cumulant generating function},
satisfies the same equation
\begin{equation}
  \label{eq:sigmamuboxtimesnu}
  \Sigma_{\mu\boxtimes\nu}(z)=\Sigma_\mu(z)\Sigma_\nu(z)
  .
\end{equation}
In \cite{Biane98} also multiplicative subordination was considered and the result asserts that for positive and free $X,Y$ there is an analytic function $F:\C^{+}\mapsto \C^{+}$, such that $\arg(F(z))\geq \arg(z)$ and
\begin{align}\label{eq:multSub}
	\E_\varphi        [
            z X^{1/2}YX^{1/2}(1-z X^{1/2}YX^{1/2})^{-1}
                    |
          X          ]=F(z)X(1-F(z)X)^{-1}.
\end{align}

The results from \cite{Biane98} were further generalized to the operator valued
setting in \cite{VoiculescuCoalgebra}.

In the present paper we provide a method for the analytic computation of the distribution of $X+f(X)Yf^*(X)$ where $X$ and $Y$ are free and $f$ is some Borel function of $X$. The main results can be summarized as follows.
\begin{enumerate}[(i)]
 \item We provide a new combinatorial formula 
  for the expectation of products of free random variables
  (Lemma~\ref{lem:intro1}) and as a corollary 
  a combinatorial formula for conditional expectations.
 \item
  The latter translates to an analytic formula
  for the conditional expectation  of the resolvent
  $\E_\varphi\left((z-X+f(X)Yf^*(X))^{-1}\|X\right)$ 
  (Theorem \ref{thm:intro1}) involving a subordination function
  similar to the one found by \cite{Biane98}.
 \item
  We find expansions
  of the free additive and free multiplicative subordination functions
  in terms of Boolean cumulants.
 \item
  The algebraic relations between the subordination
  function and various transforms give rise to 
  an algorithm for the computation of the distribution of
  $X+f(X)Yf^*(X)$. 
  In Section 6 we illustrate the algorithm with two examples,
  where explicit algebraic equations for the Cauchy transforms are found,
  from which exact analytic and combintorial information can be extracted.
\end{enumerate}
We state the announced new formulas for the conditional expectation
of alternating products of free random variables in terms of Boolean cumulants $\beta_n$
as a proposition which may be of independent interest.
\begin{proposition}\label{lem:intro1}
  Let $(X_1,\ldots,X_n)$ and $(Y_1,\ldots,Y_n)$ 
  be mutually free families in a noncommutative probability space
  $(\mathcal{A},\varphi)$, then
  \begin{equation}
    \label{eq:phiXY=phiYbetaXY}
    \varphi\left(X_1Y_1\ldots X_nY_n\right)=
    \sum_{k=0}^{n-1}\sum_{0<i_1<\ldots< i_k<n}
    \varphi\left(Y_{i_1}\ldots Y_{i_k}Y_n\right)\prod_{j=0}^{k}\beta_{2(i_{j+1}-i_j)-1}(X_{{i_j}+1},Y_{i_{j}+1},\ldots,X_{i_{j+1}}),
  \end{equation}
  where in the above sum for each fixed sequence $0< i_1<\ldots< i_k<n$ we fix $i_0=0$ and $i_{k+1}=n$.
  Consequently, if $\alg{B}$ is a subalgebra
  containing $\{X_1,X_2,\dots,X_n\}$ and free from $\{Y_1,Y_2,\dots,Y_{n-1}\}$ 
  then the conditional expectation can be written as
  \begin{equation*}
    E[X_1Y_1\ldots X_n|\alg{B}]=
    \sum_{k=0}^{n-1}\sum_{0<i_1<\ldots< i_k<n}
    Y_{i_1}\ldots Y_{i_k} \prod_{j=0}^{k}\beta_{2(i_{j+1}-i_j)-1}(X_{{i_j}+1},Y_{i_{j}+1},\ldots,X_{i_{j+1}}).
  \end{equation*}
\end{proposition}

It seems to be surprising that in the
framework of free random variables, Boolean cumulants appear quite
naturally.
However in the paper \cite{FMNS2} (see also \cite{JekelLiu:2019:operad}) 
deeper connections between freeness and Boolean cumulants are shown, in particular description of freeness in terms of Boolean cumulants and further applications are studied. We complement  proof of formula
\eqref{eq:phiXY=phiYbetaXY} with an alternative proof based on the results from \cite{FMNS2}.

The main contribution of the present paper confirms the importance of
Boolean cumulants with a combinatorial interpretation
of the coefficients of subordination functions in terms of Boolean cumulants.
While in the case of additive free convolution a probabilistic interpretation
of the coefficients of the corresponding subordination function 
is known  \cite[Proposition~4]{Woess:1986:nearest},
we are not aware of an analogous interpretation 
in the case of multiplicative free convolution.
The next corollary provides a combinatorial interpretation and in addition
an alternative series expansion in the additive case. 
The formulas from the following Corollary proved to be useful in \cite{SzpojanWesol2}.
\begin{corollary}\label{intro:Cor1}
  \begin{enumerate}[(i)]
   \item 
    If $X,Y$ are self--adjoint, bounded free random variables then the additive subordination
    function $\omega_1$ defined in \eqref{eq:addSubordination}
    has the expansion
    \begin{align*}
      \omega_1(z)=\sum_{n=0}^{\infty}\beta_{2n+1}(Y,(z-X)^{-1},\ldots,(z-X)^{-1},Y)
    \end{align*}
    in some
    neighbourhood of infinity.
   \item 
    If $X$ and $Y$ are positive  bounded free random variables then the
    multiplicative subordination function $F$ defined in \eqref{eq:multSub}
    has the expansion
    \begin{align*}
      F(z)=\sum_{n=0}^{\infty}\beta_{2n+1}(Y,X,\ldots,X,Y)z^{n+1}
    \end{align*}
    in some neighbourhood of $0$.
  \end{enumerate}
\end{corollary}

Based on the Lemma \ref{lem:intro1}, we are able to generalize the method
of \cite{Biane98} and calculate the conditional expectation of the resolvent of  $X+f(X)Yf^*(X)$.
It  can be expressed by the following remarkably simple formula.
\begin{theorem}\label{thm:intro1}
  Let $X,Y$ be free, self--adjoint, bounded random variables and assume that $f$ is a bounded
  Borel function on the spectrum of $X$, then there exists a unique function
  $\delta$ such that for $z\in \C^{+}$ one has
  
\begin{equation*}
    \E_\varphi    [\left(z-X-f(X)Yf^*(X)\right)^{-1}| \alg{B}
         ]=\left(z-X-\delta(z)f(X)f^*(X)\right)^{-1}.
  \end{equation*}
\end{theorem}
We will see that the subordination function $\delta(z)$ appearing in the theorem above can be
determined by means of free multiplicative convolutions. Consequently
Theorem~\ref{thm:intro1} gives rise to
an algorithm to compute the distribution of $X+f(X)Yf^*(X)$ for any
bounded Borel function. We will illustrate this with some examples in
Sections~\ref{sec:examples} and~\ref{sec:algebraic} below.
\begin{corollary}\label{thm:intro2}
	Let $X,Y$ be free, self--adjoint and bounded then one has
	\begin{align*}
          G_{X+f(X)Yf^*(X)}(z)=\int_{\mathbb{R}}\frac{1}{z-x-\delta(z)\abs{f(x)}^2}d\mu(x),
	\end{align*}

where $\delta(z)$ is the unique function from Theorem~\ref{thm:intro1}.
\end{corollary}

The problem of determining the distribution of $p(X_1,\ldots,X_n)$ for arbitrary
noncommutative self-adjoint polynomials $p$ in free, self-adjoint
random variables $X_1,\ldots,X_n$ was solved
in \cite{BelinschiMaiSpeicherSubordination} using the so-called
\emph{linearization trick}, which lifts the problem to that of matrix
valued additive free convolution. 
This method allows to compute numerical approximations
of the probability density function for arbitrary self-adjoint polynomials in
arbitrary free random variables $X_1,\ldots,X_n$; 
however it does now allow to extract combinatorial informations
about the moments. For the latter it is necessary to obtain more 
explicit equations.
In theory, tools from computational algebraic geometry allow to calculate
an explicit algebraic equation for the Cauchy transform of an arbitrary
polynomial $p(X_1,\ldots,X_n)$ in free random variables, provided that the
individual Cauchy transforms of the latter are algebraic.
Practice shows however that the complexity of the  algebraic systems arising
from the linearization trick exceeds the capacities of current hardware and software 
and to  our knowledge at the time of this writing no nontrivial example of an explicit
calculation has been found. 

We illustrate our result with several examples, showing that in specific cases
these technical limitations can be overcome and a description of the
distribution of functions of free variables by  means of explicit algebraic
equations for the Cauchy transforms and recurrence relations for the moments
can be obtained.

The examples include the distribution of $p+pXp$, where $p$ is a projection and
$X$ is semicircular. This is no new result, since this distribution is
identical to the distribution of $p(1+X)p$ and one can calculate it by means of free multiplicative convolution, but it is the simplest
example showing how the new method works.
We also determine the Cauchy transform of $X+XYX$ where the law of $X$ and $Y$ 
is either the semicircle or the arcsine distribution  and 
it turns out that the Cauchy transform of the resulting distribution
satisfies an algebraic
equation of degree 11, which is a new result. We devote
Section~\ref{sec:algebraic}
to an outline of this rather tedious procedure.
The method presented is fairly general in the
sense that for any polynomial $f$ and for any free random variables $X,Y$ with
algebraic Cauchy transforms one can follow the same steps to obtain an equation
for the Cauchy transform of $X+f(X)Yf^*(X)$.

Let us remark that the method presented in this paper also works in the
operator valued setting, replacing the analytic functions by operator valued functions,
see \cite{BelinschiSpeicherTreilhard:2015:operator}.

The paper is organized as follows:

In Section~\ref{sec:preliminaries} we recall some basic facts from free
probability theory, we recall the notion of conditional expectation. This
section contains also a review of the relation between free and Boolean
cumulants.

In Section~\ref{sec:boolean} we prove the basic Lemma~\ref{lem:intro1}
and we discuss the expansions of free additive and
multiplicative subordination functions announced in Corollary
\ref{intro:Cor1}.

Section~\ref{sec:subordination} contains a proof of the main Theorem~\ref{thm:intro1}.

In Section~\ref{sec:examples} we describe an algorithm based on
Theorem~\ref{thm:intro1} which allows calculate the distribution of $X+f(X)Yf^*(X)$, 
which amounts to a proof of Corollary~\ref{thm:intro2}. 
We finish this section with some explicit examples.

In Section~\ref{sec:algebraic} we calculate the distribution of $X+XYX$
for free semicircular random variables $X,Y$. There is no explicit formula
for the Cauchy transform, indeed it is an algebraic function, satisfying
an algebraic equation of order $11$.
This example can serve as a template for the calculation of an algebraic
equation satisfied the the Cauchy transform of $X+f(X)Yf^*(X)$,
starting from a polynomial $f$ and free random variables $X$ and $Y$
both having algebraic Cauchy transforms. We discuss also shortly the distribution of $X+XYX$ when $X$ and $Y$ have arcsine distribution.

\emph{Acknowledgements.}
We are grateful to V.~Vasilchuk for discussions at an early stage of this
project. He contributed an independent random matrix proof  \cite{Vasilchuk:2018:asymptotic} of the algorithm developed
in section~\ref{sec:examples} below.
We further acknowledge numerous corrections and improvements suggested
by an anonymous referee.

\section{Preliminaries}
\label{sec:preliminaries}
In this section we give a brief introduction to free probability. We mention only notions and facts that are needed in the present paper. For detailed introductions we refer to \cite{VoiDykNica,NicaSpeicherLect,MingoSpeicher}.
\subsection{Free independence}
We will work in the framework of finite von Neumann algebras, thus we assume that we
are given a von Neumann algebra $\mathcal{A}$ and a faithful, normal, tracial
state $\varphi$.  We will refer to the pair $(\mathcal{A},\varphi)$ as a
\emph{non-commutative probability space} or ncps for short. Given such a pair Voiculescu defined in \cite{VoiculescuAdd} a new notion of independence called \emph{freeness} (or \emph{free independence}). 
\begin{definition}
	We say that subalgebras $\left(\mathcal{A}_i\right)_{i\in I}$ of the ncps $\mathcal{A}$ are \emph{free} if
	\begin{align*}
	\varphi\left(a_1\cdots a_n\right)=0
	\end{align*}
	whenever
        the random variables $a_k$,  $k=1,2,\ldots,n$ are centered  with
        respect to $\varphi$, i.e., $\varphi\left(a_k\right)=0$, and
        neighbouring random variables come from different subalgebras, that is
        $a_k\in\mathcal{A}_{i_k}$ for $k=1,2,\ldots,n$ with $i_j\neq i_{j+1}$ for $j=1,2,\ldots,n-1$.
\end{definition}

The above condition allows to calculate mixed moments of free random variables
in terms of marginal ones, however the resulting expressions for mixed moments
are rather complicated.
It is thus more convenient to work with the so-called \emph{free cumulants}
\cite[Lecture 11]{NicaSpeicherLect},
which we now discuss briefly together with some of the properties which are
useful in the following.
Free cumulants are defined using the lattice of non-crossing partitions
and have the advantage that freeness can be characterized by
\emph{vanishing of mixed free cumulants}.

\subsection{Set partitions}

A \emph{partition} of the set  $[n]=\{1,\ldots,n\}$ is
a set $\pi=\{A_1,\ldots,A_k\}$ of disjoint nonempty subsets  such that
$\bigcup_{i=1}^k A_i=\{1,\ldots,n\}$.
The sets $A_1,\ldots,A_k$ are called the \emph{blocks} of $\pi$ and we write
$i\sim_\pi j$ to mean that $i,j\in[n]$ are in the same block of $\pi$.
We equip it with a lattice structure by defining a partial order $\leq$
where we define
$\pi\leq \sigma$ 
if for every block $A\in\pi$ there is a block $B\in\sigma$ such that
$A\subseteq B$. The maximal element in this lattice is the partition
consisting of only one block and it is denoted by $\hat{1}_n$.
It is customary to depict partitions as diagrams of the kind shown in
Fig.~\ref{fig:partitions}.
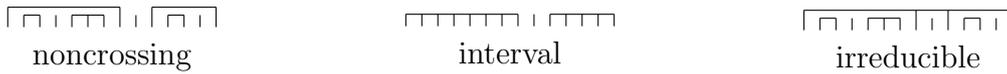
\begin{figure}
  \begin{minipage}{0.3\linewidth}
    \begin{center}
      \makeatletter{}\begin{picture}(80,6.5)(1,0)
\put(2,0){\line(0,1){7.5}}
\put(8,0){\line(0,1){4.5}}
\put(14,0){\line(0,1){4.5}}
\put(20,0){\line(0,1){4.5}}
\put(26,0){\line(0,1){4.5}}
\put(32,0){\line(0,1){4.5}}
\put(38,0){\line(0,1){4.5}}
\put(44,0){\line(0,1){7.5}}
\put(50,0){\line(0,1){4.5}}
\put(56,0){\line(0,1){7.5}}
\put(62,0){\line(0,1){4.5}}
\put(68,0){\line(0,1){4.5}}
\put(74,0){\line(0,1){4.5}}
\put(80,0){\line(0,1){7.5}}
\put(8,4.5){\line(1,0){6}}
\put(20,4.5){\line(1,0){0}}
\put(26,4.5){\line(1,0){12}}
\put(2,7.5){\line(1,0){42}}
\put(50,4.5){\line(1,0){0}}
\put(62,4.5){\line(1,0){6}}
\put(74,4.5){\line(1,0){0}}
\put(56,7.5){\line(1,0){24}}
\end{picture} 

      noncrossing
    \end{center}
  \end{minipage}
    \begin{minipage}{0.3\linewidth}
    \begin{center}
      \makeatletter{}\begin{picture}(80,3.5)(1,0)
\put(2,0){\line(0,1){4.5}}
\put(8,0){\line(0,1){4.5}}
\put(14,0){\line(0,1){4.5}}
\put(20,0){\line(0,1){4.5}}
\put(26,0){\line(0,1){4.5}}
\put(32,0){\line(0,1){4.5}}
\put(38,0){\line(0,1){4.5}}
\put(44,0){\line(0,1){4.5}}
\put(50,0){\line(0,1){4.5}}
\put(56,0){\line(0,1){4.5}}
\put(62,0){\line(0,1){4.5}}
\put(68,0){\line(0,1){4.5}}
\put(74,0){\line(0,1){4.5}}
\put(80,0){\line(0,1){4.5}}
\put(2,4.5){\line(1,0){42}}
\put(50,4.5){\line(1,0){0}}
\put(56,4.5){\line(1,0){24}}
\end{picture} 

      interval
    \end{center}
  \end{minipage}
    \begin{minipage}{0.3\linewidth}
    \begin{center}
      \makeatletter{}\begin{picture}(80,6.5)(1,0)
\put(2,0){\line(0,1){7.5}}
\put(8,0){\line(0,1){4.5}}
\put(14,0){\line(0,1){4.5}}
\put(20,0){\line(0,1){4.5}}
\put(26,0){\line(0,1){4.5}}
\put(32,0){\line(0,1){4.5}}
\put(38,0){\line(0,1){4.5}}
\put(44,0){\line(0,1){7.5}}
\put(50,0){\line(0,1){4.5}}
\put(56,0){\line(0,1){7.5}}
\put(62,0){\line(0,1){4.5}}
\put(68,0){\line(0,1){4.5}}
\put(74,0){\line(0,1){4.5}}
\put(80,0){\line(0,1){7.5}}
\put(8,4.5){\line(1,0){6}}
\put(20,4.5){\line(1,0){0}}
\put(26,4.5){\line(1,0){12}}
\put(50,4.5){\line(1,0){0}}
\put(62,4.5){\line(1,0){6}}
\put(74,4.5){\line(1,0){0}}
\put(2,7.5){\line(1,0){78}}
\end{picture} 

      irreducible
    \end{center}
  \end{minipage}
  \caption{Partition diagrams}
  \label{fig:partitions}
\end{figure}

We will work with two sublattices, namely the lattices of
\emph{noncrossing partitions} and \emph{interval partitions}.

A partition $\pi$ is called \emph{noncrossing} if any ordered quadruple
$p_1<q_1<p_2<q_2$ cannot satisfy $p_1\sim_\pi p_2$ and $q_1\sim_\pi q_2$ unless
$p_1,p_2,q_1,q_2$ are in the same block of $\pi$.

We denote the set of all non-crossing partitions of $[n]$  by $\NC(n)$
and it can be shown that the restriction of the partial order defined above turns
it into a sublattice.

An \emph{interval partition}  is a partition $\pi$ of $[n]$ 
such that every block is an interval, i.e.,  $i\sim_\pi k$ and $i<j<k$  then
$i,j,k$ are in the same block of $\pi$. 
The set of all interval partitions is denoted by $\IP(n)$.
Again the restriction of the partial order defined above turns it into a sublattice.

\subsection{Another partial order on noncrossing partitions}
The lattice of interval partitions $\IP(n)$ is a sublattice of the lattice of
noncrossing partitions $\NC(n)$. We can thus define the \emph{interval closure}
$\hat{\pi}\in\IP(n)$ of a noncrossing partition $\pi\in\NC(n)$ as the smallest
interval partition which dominates $\pi$.

A noncrossing partition $\pi\in\NC(n)$ is called \emph{irreducible} if
$\hat{\pi}=\hat{1}_n$;
combinatorially this means that $\pi$ is irreducible if $1\sim_\pi n$.
We denote the set of irreducible noncrossing partitions by $\NCirr(n)$.
Any partition can be written uniquely as a \emph{concatenation} of irreducible
partitions. For example, the first partition in Fig.~\ref{fig:partitions} is
the concatenation $\pi=\pi_1\pi_2\pi_3$ where
$\pi_1=
\makeatletter{}\begin{picture}(44,6.5)(1,0)
\put(2,0){\line(0,1){7.5}}
\put(8,0){\line(0,1){4.5}}
\put(14,0){\line(0,1){4.5}}
\put(20,0){\line(0,1){4.5}}
\put(26,0){\line(0,1){4.5}}
\put(32,0){\line(0,1){4.5}}
\put(38,0){\line(0,1){4.5}}
\put(44,0){\line(0,1){7.5}}
\put(8,4.5){\line(1,0){6}}
\put(20,4.5){\line(1,0){0}}
\put(26,4.5){\line(1,0){12}}
\put(2,7.5){\line(1,0){42}}
\end{picture} 
$,
$\pi_2=
\makeatletter{}\begin{picture}(2,3.5)(1,0)
\put(2,0){\line(0,1){10}}
\put(2,10){\line(1,0){0}}
\end{picture} 
$
and
$\pi_3=
\makeatletter{}\begin{picture}(26,6.5)(1,0)
\put(2,0){\line(0,1){7.5}}
\put(8,0){\line(0,1){4.5}}
\put(14,0){\line(0,1){4.5}}
\put(20,0){\line(0,1){4.5}}
\put(26,0){\line(0,1){7.5}}
\put(8,4.5){\line(1,0){6}}
\put(20,4.5){\line(1,0){0}}
\put(2,7.5){\line(1,0){24}}
\end{picture} 
$.

Belinschi and Nica  \cite{BelinschiNicaBBPforKTuples} defined a coarsening of
the usual partial order on noncrossing partitions by defining
$\pi\ll\sigma$ if and only if $\pi \leq \sigma$ and
for any block $B\in\sigma$ there is a block $C\in\pi$ such that
$\min(B)=\min(C)$ and $\max(B)=\max(C)$.
That is, if for every block $S\in\sigma$ the restrictions of $\pi|_S$ is irreducible.
In particular a partition $\pi$ is irreducible if and only if $\pi\ll\hat{1}_n$.

\subsection{Free cumulants}
Using the lattice of non-crossing partitions one can define the so called
\emph{free cumulants} $\kappa_n$.
To do so we introduce the following notation.
Given a sequence of multilinear functionals  $a_n:\mathcal{A}^n\mapsto \C$, $n=1,2,\dots$,
and a partition $\pi$ of $[n]$ we define the partitioned functional $a_\pi$ by
setting
\begin{equation*}
  a_\pi(X_1,\ldots, X_n)=\prod_{B\in \pi} a_{|B|}\left(X_1,\ldots, X_n|B\right),
\end{equation*}
where
\begin{equation*}
  a_{|B|}\left(X_1,\ldots, X_n|B\right)=a_{|B|}\left( X_i;i\in B\right).   
\end{equation*}
Then the \emph{free cumulants} are uniquely determined by the recursive equations
\begin{equation}
  \label{eq:deffreecumulants}
  \varphi(X_1\cdots X_n)=\sum_{\pi\in \NC(n)}\kappa_\pi(X_1,\ldots,X_n).
\end{equation}
More generally we then have
\begin{align*}
  \varphi_\sigma(X_1,\ldots, X_n)
  =\sum_{\substack{\pi\in \NC(n)\\
         \pi\leq\sigma
        }}
    \kappa_\pi(X_1,\ldots,X_n)
\end{align*}
where
\begin{equation*}
  \varphi_n(X_1,\ldots, X_n)=\varphi(X_1X_2\dotsm X_n).
\end{equation*}
It turns out that freeness can be characterized in terms of free cumulants as
follows \cite[Theorem~11.16]{NicaSpeicherLect}:
Subalgebras $\mathcal{A}_i\subseteq\mathcal{A}$ are free
if  $\kappa_k\left(X_1,X_2,\ldots,X_n\right)=0$ whenever $n\geq2$,
each $X_j$ lies in one of the subalgebras and at least two different subalgebras appear.

\subsection{Boolean cumulants}
Boolean cumulants are defined analogously by replacing
the lattice of noncrossing partitions
in   \eqref{eq:deffreecumulants}  by the lattice of interval partitions:
\begin{equation*}
  \varphi(X_1\cdots X_n)=\sum_{\pi\in \IP(n)}\beta_\pi(X_1,\ldots,X_n).
\end{equation*}
The vanishing of Boolean cumulants characterizes another instance of
non-commutative independence called \emph{Boolean independence}.
In the present paper we are not interested in Boolean independence and
therefore skip the definition. One of our main results is that Boolean cumulants are appear naturally in calculation of conditional expectations of functions of free random variables. 

\subsection{Relation between free and Boolean cumulants}
It follows from the M\"obius inversion formula that free and Boolean cumulants
determine each other and there is 
an explicit formula, see
\cite{Lehner:2002:connected,BelinschiNicaBBPforKTuples,ArizmendiHasebeLehnerVargas:2015:relations}:
\begin{equation}\label{eq:betaKappaconnection}
  \beta_n(X_1,\ldots,X_n)=
  \sum_{\pi\in\NCirr(n)}
  \kappa_\pi(X_1,\ldots,X_n)
\end{equation}
and more generally for $\pi\in NC(n)$
\begin{equation*}
  \beta_\pi(X_1,\ldots,X_n)=
  \sum_{\substack{\rho\in \NC(n)\\  \rho\ll\pi}}
  \kappa_\rho(X_1,\ldots,X_n).
\end{equation*}

\subsection{Generating functions}
The combinatorial relations discussed above give rise to functional relations
between various generating functions. In the present paper
the following functions will play a role.
\begin{enumerate}[1.]
 \item The \emph{moment generating function} of a random variable $X$ is the function
  $$
  \mgfpsi_X(z) = \sum_{n=1}^\infty \varphi(X^n)z^n.
  $$
 \item In the algebraic computations of Section~\ref{sec:algebraic} it will be
  more convenient to consider the \emph{augmented moment generating function}
  $$
  \mgf_X(z) = \sum_{n=0}^\infty \varphi(X^n)z^n = 1 + \mgfpsi_X(z).
  $$
 \item The \emph{Boolean cumulant generating function}
  $$
  \eta_X(z) = \sum_{n=1}^\infty \beta_n(X)z^n
  $$
  satisfies the relation
  $$
  \mgf_X(z) = \frac{1}{1-\eta_X(z)}
  $$
  i.e.,
  $$
  \eta_X(z) = \frac{\mgfpsi_X(z)}{1+\mgfpsi_X(z)}.
  $$
 \item The \emph{shifted Boolean cumulant generating function}
  \begin{equation}
    \label{eq:scgf}
    \tilde{\eta}_X(z) = \sum_{n=1}^\infty \beta_n(X)z^{n-1} = \frac{\eta_X(z)}{z}.
  \end{equation}
 \item Given two random variables $X$ and $Y$, the \emph{alternating Boolean
    cumulant generating function} 
  \begin{equation}
    \label{eq:abcgf}
    \abcgf{Y}{X}(z) = \sum_{n=0}^\infty \beta_{2n+1}(Y,X,Y,X,\dots,X,Y)z^{2n+1}
  \end{equation}
  will play a central role. As we shall see below (see
  Corollary~\ref{cor:multsubord}), it is essentially
  the multiplicative subordination function~\eqref{eq:multSub}.
\end{enumerate}

\subsection{Conditional expectations in von Neumann algebras}
In this subsection we briefly recall the notion of conditional expectation in
von Neumann algebras. For more details we refer to \cite{Takesaki}.

Assume that $(\mathcal{A},\varphi)$ is a  $W^*$-probability space, i.e.,
$\mathcal{A}$ is a finite von Neumann algebra and $\varphi$ a faithful, normal,
tracial state. Then for any von Neumann subalgebra $\mathcal{B}\subset
\mathcal{A}$  there exists a faithful, normal projection
$\E_\mathcal{B}:\mathcal{A}\to \mathcal{B}$,  called the \emph{conditional
  expectation} onto the subalgebra $\mathcal{B}$ with respect to $\varphi$, such that
$\varphi\circ\E_\mathcal{B}=\varphi$.
It is a unital $\alg{B}$-module map, i.e.,
$E_\alg{B}(Y_1XY_2)=Y_1E_\alg{B}(X)Y_2$, for all $X\in\mathcal{A}$ and $Y_1,Y_2\in\mathcal{B}$. 
In other words,
$\E_\mathcal{B}(X)$ is the unique element $Z\in\alg{B}$ such that for any
$Y\in\mathcal{B}$ 
one has $\varphi(XY)=\varphi(ZY)$.
For a fixed element $X\in\mathcal{A}$ we denote by $\E_X$ 
the conditional expectation onto the von Neumann subalgebra generated by $X$.

\section{Boolean cumulants and subordination functions}
\label{sec:boolean}

In this section we prove a lemma which plays a crucial role in this paper.
It is essentially the same as  Biane's subordination result for free
multiplicative convolution~\cite[Proposition~3.6]{Biane98}. One can follow the steps of Biane's proof with one additional summation at the and. However we present here a shorter proof.
In addition we reveal a combinatorial interpretation of the subordination function
in terms of Boolean cumulants, which was not
of interest in \cite{Biane98}, but which provides combinatorial information
about the subordination function which is essential in the sequel.  For another
interpretation of the subordination function as generating function of certain
``taboo'' probabilities in the context of random walks see
\cite[Proposition~4]{Woess:1986:nearest}.
We also present a sketch of an alternative proof which uses recent
characterization of freeness in terms of Boolean cumulants given in \cite{FMNS2,JekelLiu:2019:operad}. We then use this lemma to give a precise formula for the
conditional expectation of certain resolvents.
To this end we first expand alternating joint moments
free random variables in terms of moments of the first
and mixed Boolean cumulants.

\subsection{Yet another formula for expectations of free random variables}
\begin{lemma}[$X$ and $Y$ exchanged]
\label{prop:31}
  Let $(X_1,\ldots,X_n)$ and $(Y_1,\ldots,Y_n)$ be mutually free families in a
  noncommutative probability space $(\mathcal{A},\varphi)$, then
  \begin{multline}    \varphi\left(Y_1X_1 Y_2X_2\ldots Y_nX_n\right)
    \\
    =
    \sum_{k=0}^{n-1}\sum_{0<i_1<\ldots< i_k<n}
    \varphi\left(X_{i_1}\ldots X_{i_k}X_n\right)\prod_{j=0}^{k}\beta_{2(i_{j+1}-i_j)-1}(Y_{{i_j}+1},X_{i_{j}+1},\ldots,Y_{i_{j+1}}),
  \end{multline}
  where in the above sum for each fixed sequence $0< i_1<\ldots< i_k<n$
  we fix $i_0=0$ and $i_{k+1}=n$. 
\end{lemma}
\makeatletter{}\begin{proof}[Proof 1]
  We start with the cumulant expansion \eqref{eq:deffreecumulants}
  \begin{align*}
    \varphi\left(Y_1X_1 Y_2X_2\ldots Y_nX_n\right)
    &= \sum_{\pi\in\NC(2n)}\kappa_\pi(Y_1,X_1,Y_2,X_2,\dots,Y_n,X_n);
      \intertext{
      now mixed cumulants vanish and
      if we relabel the set $\{1,2,\dots,2n\}$
      to $\{1',1'',2',2'',\dots,n',n''\}$
      we can decompose
      every contributing partition into two
      parts $\pi'\in\NC(\{1,3,\dots,2n-1\})\simeq \NC(\{1',2',\dots,n'\})$ and
      $\pi''\in\NC(\{2,4,\dots,2n\})\simeq\NC(\{1'',2'',\dots,n''\})$ 
      }
    &= \sum_{\substack{\pi',\pi''\\
             \pi'\cup\pi''\in\NC(2n)}}
      \kappa_{\pi'}(Y_1,Y_2,\dots,Y_n)
      \,
      \kappa_{\pi''}(X_1,X_2,\dots,X_n)
  \end{align*}
  Now assume that $\pi'$ has $k$ outer blocks,
  say $B_1,B_2,\dots, B_k$ and let $I_j=\overline{B_j}$
  be their convex hulls.
  Then the restrictions $\pi_j=\pi|_{I_j}$ are irreducible
  (the unique outer block being $B_j$ from $\pi'$) and we can decompose
  $\pi=\pi_0\cup\pi_1\dots\pi_k$
  where $\pi_0\in\NC(B_0)$ is a noncrossing partition
  of the complement $B_0=[2n]\setminus
  \left(\overline{B_1}\cup\overline{B_2}\cup\dots\cup \overline{B_k}\right)$.
  Note that 
  $2n\in B_0$.
  All sets in this construction are uniquely determined,
  see the example in Fig.~\ref{fig:partition123213etc}
  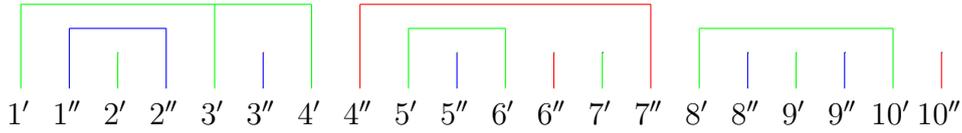
\begin{figure}
    \centering
    \makeatletter{}\begin{tikzpicture}[scale=\textwidth/160cm]
  \foreach\i in {1,...,10}
    \coordinate[color=green,label=center:$\i'$] (p) at (-10.3+12*\i,-3);
  \foreach\i in {1,...,10}
    \coordinate[color=blue,label=center:$\i''$] (p) at (-4.3+12*\i,-3);
\draw[color=green] (2,0)--(2,10.5);
\draw[color=blue] (8,0)--(8,7.5);
\draw[color=green] (14,0)--(14,4.5);
\draw[color=blue] (20,0)--(20,7.5);
\draw[color=green] (26,0)--(26,10.5);
\draw[color=blue] (32,0)--(32,4.5);
\draw[color=green] (38,0)--(38,10.5);
\draw[color=red] (44,0)--(44,10.5);
\draw[color=green] (50,0)--(50,7.5);
\draw[color=blue] (56,0)--(56,4.5);
\draw[color=green] (62,0)--(62,7.5);
\draw[color=red] (68,0)--(68,4.5);
\draw[color=green] (74,0)--(74,4.5);
\draw[color=red] (80,0)--(80,10.5);
\draw[color=green] (86,0)--(86,7.5);
\draw[color=blue] (92,0)--(92,4.5);
\draw[color=green] (98,0)--(98,4.5);
\draw[color=blue] (104,0)--(104,4.5);
\draw[color=green] (110,0)--(110,7.5);
\draw[color=red] (116,0)--(116,4.5);
\draw[color=green] (14,4.5)--(14,4.5);
\draw[color=blue] (8,7.5)--(20,7.5);
\draw[color=red] (32,4.5)--(32,4.5);
\draw[color=green] (2,10.5)--(38,10.5);
\draw[color=green] (56,4.5)--(56,4.5);
\draw[color=green] (50,7.5)--(62,7.5);
\draw[color=green] (68,4.5)--(68,4.5);
\draw (74,4.5)--(74,4.5);
\draw[color=red] (44,10.5)--(80,10.5);
\draw (92,4.5)--(92,4.5);
\draw[color=green] (98,4.5)--(98,4.5);
\draw (104,4.5)--(104,4.5);
\draw[color=green] (86,7.5)--(110,7.5);
\draw (116,4.5)--(116,4.5);
\end{tikzpicture} 
    \caption{A partition illustrating the proof of Lemma~\ref{prop:31}:
      $\pi'$ is green and has four irreducible components,
      $\pi''$ is blue and red, the latter marking the
      distinguished componet $\pi_0$}
    \label{fig:partition123213etc}
  \end{figure}
  and we
  can thus rearrange the sum.
  \begin{multline*}
    \varphi\left(Y_1X_1 Y_2X_2\ldots Y_nX_n\right)\\
  \begin{aligned}
    &= \sum_{\substack{B_0\subseteq\{2,4,\dots,2n\}\\ B_0=\{2i_1,2i_2,\dots,2i_k,
        2n\}}}
    \sum_{\pi_0\in\NC(B_0)} \kappa_{\pi_0}(X_{i_1},X_{i_2},\dots,X_{i_k},X_n)
    \prod_{j=1}^k
    \sum_{\pi_j\in\NCirr(\overline{B_j})}
    \kappa_{\pi_j}(Y_{i_{j-1}+1},X_{i_{j-1}+1},\dots,Y_{i_j-1})
    \\
    &= \sum_{\substack{B_0\subseteq\{2,4,\dots,2n\}\\ B_0=\{2i_1,2i_2,\dots,2i_k,
        2n\}}}
    \varphi(X_{i_1}X_{i_2}\dots X_{i_k}X_n)
    \prod_{j=1}^k
    \beta(Y_{i_{j-1}+1},X_{i_{j-1}+1},\dots,Y_{i_j-1})
  \end{aligned}
  \end{multline*}
  where relation \eqref{eq:betaKappaconnection} was used.

\end{proof}

\begin{proof}[Sketch of a proof based on \cite{FMNS2}]
  First observe that from \cite[Theorem 1.2]{FMNS2} it follows that a mixed
  Boolean cumulant of $\{Y_1,\ldots,Y_n\}$ and $\{X_1,\ldots,X_n\}$ vanishes
  whenever the first variable is free from the last one. Thus in the expansion 
  \begin{align*}
    \varphi\left(Y_1X_1Y_2X_2\cdots Y_nX_n\right)=
    \sum_{\pi\in\IP(2n)} \beta_{\pi}(Y_1,X_1,\ldots,Y_n,X_n)
  \end{align*}
  all partitions with blocks starting from one of the $X$'s and ending at one
  of the $Y$'s (or vice versa) do not contribute.
	
  Next we expand each block of $\beta_{\pi}(Y_1,X_1,\ldots,Y_n,X_n)$ according
  to \cite[Theorem 1.2]{FMNS2}. After this expansion we reorganize the sum in the
  following way: we pick $X_{i_1},\ldots,X_{i_k},X_n$ and consider only those
  partitions whose outer block contains exactly these $X$'s. Next we observe
  that each such choice corresponds to an interval partition $\pi\in \IP(k+1)$
  on $X_{i_1},\ldots,X_{i_k},X_n$ which fixes the outer block. The final
  observation is the following: using Theorem 1.2 from \cite{FMNS2} one can check that
  regardless if $X_{i_j}$ and $X_{i_{j+1}}$ are in the same block or not, the
  sum restricted to $Y_{i_j+1},X_{i_j+1},\ldots,Y_{i_{j+1}}$ always
  gives \[\beta_{2(i_{j+1}-i_j)-1}(Y_{{i_j}+1},X_{i_{j}+1},\ldots,Y_{i_{j+1}}).\] 
	On the other hand summing over all choices of interval partitions on $X_{i_1},\ldots,X_{i_k},X_n$ (corresponding to outer blocks of $Y$'s) we get
	\[\varphi\left(X_{i_1}\cdots X_{i_k}X_n\right),\]
	which completes the proof.
\end{proof}

The following corollary is an immediate consequence of Lemma~\ref{prop:31}
and will provide the basis for the further calculations.
\begin{corollary}
  \label{cor:condexp}
  Let $(\mathcal{A},\varphi)$ be a W$^*$-probability space and
  $\mathcal{B}$  a von Neumann subalgebra.
  Assume that $(X_1,\ldots,X_n)$,  $(Y_1,\ldots,Y_n)$ are two families such that
  $(X_1,\ldots,X_n)\subseteq\mathcal{B}$ and
  $(Y_1,\ldots,Y_n)$ is free from $\mathcal{B}$.
  Then the conditional expectation of alternating monomials can be evaluated as follows
	\begin{multline}\label{eq:EYXY}
          \E_\varphi\left[Y_1X_1Y_1\ldots X_{n-1}Y_n | \mathcal{B}\right]
          \\
          =
	\sum_{k=0}^{n-1}\sum_{1\leq i_1<\ldots< i_k\leq n-1}
	X_{i_1}\ldots X_{i_k}\prod_{j=0}^{k}\beta_{2(i_{j+1}-i_j)-1}(Y_{{i_j}+1},X_{i_{j}+1},\ldots,Y_{i_{j+1}}),
	\end{multline}
	where in the above sum for each sequence $0< i_1<\ldots< i_k<n$ we set $i_0=0$ and $i_{k+1}=n$.
\end{corollary}
For a different combinatorial approach to conditional expectations,
which avoids mixed cumulants, see \cite{Cebron:2013:free}.
On the other hand, this can be seen as another instance of the unshuffle coproduct
\cite[Definition~3.3]{EbrahimiFardPatras:2015:halfshuffles} appearing
in free probability; 
this connection will be investigated elsewhere.

The expansion \eqref{eq:EYXY} shows that for an alternating monomial
the conditional expectation
is in fact contained in the subalgebra generated by  $X_i$.
By standard arguments based on the Kaplansky density theorem
\cite[Theorem~44.1]{Conway:2000:course} and the weak* continuity
of the conditional expectation show that this property
passes to the strong closure.
\begin{corollary}
  \label{cor:kaplansky}
  Let $(\mathcal{A},\varphi)$ be a W$^*$-probability space and
  $\alg{C}\subseteq \alg{B}\subseteq \alg{A}$ nested von Neumann subalgebras.
  Given two families
   $(X_i)_{i\in I}$ and $(Y_j)_{i\in J}\subseteq \alg{A}$ such that
   $(X_i)_{i\in I}\subseteq \alg{C}$ and 
   $(Y_j)_{i\in J}$ is free from $\alg{B}$,
   and an element
   $Z\in( \{X_i \mid i\in I\}\cup \{Y_j \mid j\in J\})''$,
   we have
   $$
   \E[Z | B] \in \alg{C}.
   $$
\end{corollary}

\subsection{A formula for conditional expectations of  resolvents}
The next lemma is the main result of this section.

\begin{lemma} \label{lem:31}
  Let $(\alg{A},\varphi)$ be an $W^*$--probability space and $\alg{B}\subseteq\alg{A}$ a von Neumann subalgebra. 
  Assume that an element  $Y\in\alg{A}$ is free from $\alg{B}$ 
  and pick an element $X\in\alg{B}$ such that  $\norm{Y}\norm{X}<1/5$. 
  Then the conditional expectation of the resolvent $(1-YX)^{-1}Y$
  is a resolvent again; more precisely
  \begin{align}\label{eq:MainLem}
    \E_\varphi\left[(1-YX)^{-1}Y | \alg{B} \right]=\eta^{Y}_{X}(1)(1-\eta^{Y}_{X}(1) X)^{-1},
  \end{align}
  where $\eta^{Y}_{X}(z)$ is the generating function
        \eqref{eq:abcgf} of
  alternating Boolean cumulants.
\end{lemma}
\begin{proof}
  Let us first record the simple estimate
  \begin{align*}
    \left|\beta_{2n+1}(Y,X,\ldots,X,Y)\right|=\biggl|\sum_{\pi\in \IP(2n+1)}(-1)^{|\pi|+1}\varphi_{\pi}(Y,X,\ldots,X,Y)\biggr|\leq 2^{2n}\norm{Y}^{n+1}\norm{X}^n
        .
  \end{align*} 
  Together with our assumption $\norm{X}\norm{Y}<1/5$ this implies that
  the series for $\eta^{Y}_{X}(z)$ converges at $z=1$ and moreover the estimate
  $|\eta_{X}^{Y}(1)|\leq\frac{\norm{Y}}{1-4\,\norm{Y}\,\norm{X}}$ holds.
  
  The series $(1-YX)^{-1}=\sum_{n=0}^{\infty}(YX)^n$ converges absolutely and
  it remains to calculate   the $\alg{B}$-valued conditional expectations
  of the individual terms. From Corollary~\ref{cor:condexp} we infer
  \begin{equation}
    \label{eq:EphiYXn=sumkpi}
    \begin{aligned}
      \E_\varphi\left[(YX)^nY | \alg{B}\right]
      &= \sum_{k=0}^{n}\sum_{1\leq i_1<\ldots< i_k\leq n}
          X^k\prod_{j=0}^{k}\beta_{2(i_{j+1}-i_j)-1}(Y,X,Y,\ldots,X,Y)
      \\
      &= \sum_{k=0}^{n}\sum_{m_0+m_1+\dotsm+m_k=n-k}
          X^k\prod_{j=0}^{k}\beta_{2m_j+1}(Y,X,Y,\ldots,X,Y).
    \end{aligned}
  \end{equation}
  Since $|\eta_{X}^{Y}(1)|\leq\frac{\norm{Y}}{1-4\,\norm{Y}\,\norm{X}}$, 
  and $\norm{X}\norm{Y}<1/5$ we have $\norm{\eta_{X}^{Y}(1)X}<1$
  we can sum the identity
  \eqref{eq:EphiYXn=sumkpi}
  over $n\geq 0$
  and  obtain
  \begin{align*}
    \E_\varphi\left[(1-YX)^{-1}Y|\alg{B}\right]
    &= \sum_{k=0}^\infty \sum_{m_0,m_1,\dots,m_k=0}^\infty X^k
      \prod_{j=0}^k\beta_{2m_j+1}(Y,X,Y,\dots,X,Y)\\
    &=(1-\eta^{Y}_{X}(1) X)^{-1}\eta^{Y}_{X}(1)
  \end{align*}

  which completes the proof of the lemma.
\end{proof}

\begin{corollary}
  In the setting of Lemma~\ref{lem:31},
  \begin{align}\label{eq:E(1-XY)X}
    \E_\varphi\left[(1-XY)^{-1}X | \alg{B}\right]=X(1-\eta^{Y}_{X}(1) X)^{-1}.
  \end{align}
\end{corollary}
\begin{proof}
  The proof consists of a rearrangement of the formula in terms of the previous one:
  \begin{align*}
    \E[ (1-XY)^{-1}X| \alg{B}]
    &= \E[X(1-YX)^{-1}| \alg{B}]\\
    &= X\E[1+(1-YX)^{-1}YX| \alg{B}]\\
    &= X + X\E[(1-YX)^{-1}Y| \alg{B}]X\\
    &= X + X \eta^{Y}_{X}(1)(1-\eta^{Y}_{X,Y}(1) X)^{-1} X\\
    &= X(1 +(1-\eta^{Y}_{X}(1) X)^{-1}  \eta^{Y}_{X}(1)X\\
    &= X(1-\eta^{Y}_{X}(1) X)^{-1}.
  \end{align*}
\end{proof}

\subsection{Subordination for multiplicative and additive free convolutions}
We finish this section by showing how subordination of free additive and 
multiplicative convolution may be deduced from Lemma \ref{lem:31}.
For the
multiplicative convolution this is straightforward, however our approach to
additive convolution (Example~\ref{ex:addConv} below) seems to be new. 
In the next section we will generalize
this observation and obtain a method of calculating the distribution of $X+f(X)Yf^*(X)$, for $X,Y$ free.

\begin{corollary}[{Subordination for multiplicative free convolution}]
\label{cor:multsubord} 
Let $T$ be free from $\alg{B}$,  $S\in \alg{B}$  and 
assume that both  $S$ and $T$ are positive.
Then for small $z$ the multiplicative subordination function from \eqref{eq:multSub} is
given by the convergent series
\begin{align*}
F(z)&= \eta_S^{zT}(1)\\
    &= \sum_{n=0}^{\infty}\beta_{2n+1}(T,S,\ldots,S,T)z^{n+1}
      .
\end{align*}
Consequently, the alternating Boolean cumulant generating function     \eqref{eq:abcgf}
is given by
\begin{equation*}
  \eta_X^Y(z) = \frac{1}{z}F(z^2)
  .
\end{equation*}
\end{corollary}
\begin{proof}
  Substitute $X=S$ and $Y=z T$ into equation \eqref{eq:MainLem} 
  for $z\in \mathbb{C}\setminus \mathbb{R}_+$. 
  After multiplying both sides by $S$ from the left one gets 
\begin{equation*}
  \E_\varphi\left[(1-zST)^{-1}zST | \alg{B}\right]=\eta^{zT}_{S}(1)S(1-\eta^{zT}_{S}(1) S)^{-1}.
\end{equation*}
Applying $\varphi$ to both sides of the above equation we get the following
relation for the moment generating functions:
\begin{equation}
  \label{eq:MSTz}
  \mgfpsi_{ST}(z)=\mgfpsi_S(\eta_{S}^{zT}(1))
  .
\end{equation}
On the other hand we know from \cite{Biane98} that there is a unique
analytic function $\omega$  on $\mathbb{C}\setminus\mathbb{R}_+$
satisfying the relation 
\begin{equation*}
  \mgfpsi_{ST}(z)=\mgfpsi_S(\omega(z)).
\end{equation*}
The  function $\omega$ is called the \emph{multiplicative subordination
  function} and since for $\abs{z}$ small the series
\begin{equation*}
  \eta^{zT}_{S}(1)=\sum_{n=0}^{\infty}\beta_{2n+1}(T,S,\ldots,S,T)z^{n+1},
\end{equation*}
converges,
we conclude that the multiplicative subordination function $\omega(z)$ is its analytic continuation.  
\end{proof}

\begin{corollary}[{Subordination for additive free convolution}]
\label{cor:additivesubordination}  
Let $X$ and $Y$ be selfadjoint free random variables, 
then for large $z$ the subordination function
 $\omega_1(z)$ from \eqref{eq:addCondExp} is given by the convergent series 
 \begin{equation*}
   \omega_1(z)=z-\sum_{n=0}^{\infty}\beta_{2n+1}(Y,(z-X)^{-1},Y,\ldots,(z-X)^{-1},Y).
 \end{equation*}

\end{corollary}
\begin{proof}
  Observe that
\begin{equation*}
  (z-X-Y)^{-1}=\left(1-(z-X)^{-1}Y\right)^{-1}(z-X)^{-1}
  .
\end{equation*} 
Let $X\in \alg{B}$ and assume that $Y$ is free from $\alg{B}$, then for $z$ large enough we can apply
\eqref{eq:E(1-XY)X} and get
\begin{equation*}
  \E_\varphi\left[(z-X-Y)^{-1} |
  \alg{B}\right]=(z-\eta^{Y}_{(z-X)^{-1}}(1)-X)^{-1}.
\end{equation*} 
Applying $\varphi$ to both sides of the above equation we get that
\[G_{X+Y}(z)=G_X\left(z-\eta^{Y}_{(z-X)^{-1}}(1)\right),\]
where
\begin{equation*}
  \eta^{Y}_{(z-X)^{-1}}(s)=\sum_{n=0}^{\infty}\beta_{2n+1}(\underbrace{Y,(z-X)^{-1},Y,\ldots,(z-X)^{-1},Y}_{
    \mbox{$2n+1$  arguments}})s^{2n+1}.
\end{equation*}
This means that  the additive subordination function $\omega_1$  is an analytic
continuation of
the function $H$ defined in some neighbourhood of infinity by the convergent series
\begin{equation*}
  H(z)=z-\sum_{n=0}^{\infty}\beta_{2n+1}(\underbrace{Y,(z-X)^{-1},Y,\ldots,(z-X)^{-1},Y}_{\mbox{$2n+1$ arguments}}).
\end{equation*}
\end{proof}

\section{Subordination for $X+f(X)Yf^*(X)$}
\label{sec:subordination}
In this section we show a new subordination result for free convolutions, 
namely an explicit formula for conditional expectation of the resolvent of
$X+f(X)Yf^*(X)$ for arbitrary Borel functions $f$.
We also present a method to  determine the subordination function enabling us
to calculate explicitly the distribution of $X+f(X)Yf^*(X)$. We assume for the whole section that $f$ is not constantly zero on the spectrum of $X$.

\begin{theorem}\label{thm:41}
  Let $(\alg{A},\varphi)$ be a $W^*$ probability space. Assume that $\alg{B}\subset\alg{A}$ is a von Neumann subalgebra. Let $X,Y$ be self--adjoint, such that $X\in\alg{B}$ and $Y$ is free from $\alg{B}$. Moreover assume that $f$ is a bounded Borel function on the spectrum of
  $X$, then there exists a function $\delta$ such that
  \begin{enumerate}[(i)]
   \item \label{it:mainthm:condexp}
    The function $\delta(z)$ is a subordination function in the sense that
    \begin{equation}
      \label{eq:mainthm:condexp}
    \E_\varphi\left[\left.\left(z-X-f(X)Yf^*(X)\right)^{-1}\right| \mathcal{B}\right]=\left(z-X-\delta(z)f(X)f^*(X)\right)^{-1}
    \end{equation}
    for $z\in \IC^+$.
   \item \label{it:mainthm:delta} 
    The function
    $\delta:\mathbb{C}^+\mapsto\mathbb{C}^-\cup\mathbb{R}$  is analytic and has
    nontangential limit
    $$
    \nontang{z\to\infty} \frac{\delta(z)}{z}= 0
    .
    $$
   \item \label{it:mainthm:deltaequation1}
    The functional equation
    \begin{equation}
       \label{eq:mainthm:deltaequation1}
      \mgfpsi_{f(X)(z-X)^{-1}f(X)^*Y}(1)=\mgfpsi_{f(X)(z-X)^{-1}f(X)^*}\left(\delta(z)\right)
    \end{equation}
    holds for $z\in\mathbb{C}^{+}$. 
   \item \label{it:mainthm:deltaequation2}
    Equivalently, the function $\delta(z)$ satisfies the fixed point equation
    \begin{equation}
       \label{eq:mainthm:deltaequation2}      
    \widetilde{\eta}_Y\left(\widetilde{\eta}_{f(X)(z-X)^{-1}f(X)^*}(\delta(z))\right)=\delta(z),      
    \end{equation}
    where $ \widetilde{\eta}(z)$ is the shifted Boolean cumulant generating
    function \eqref{eq:scgf}.
   \item The function $\delta$ is uniquely determined by
    \eqref{it:mainthm:delta} and \eqref{it:mainthm:deltaequation1},
    and by     \eqref{it:mainthm:delta} and \eqref{it:mainthm:deltaequation2}, respectively.
  \end{enumerate}

\end{theorem}
We split the proof of the above theorem into several propositions.
First we establish  \eqref{it:mainthm:condexp}.
\begin{proposition}
  \label{prop:EzXfXYfX*}
  Let $X,Y$ be bounded selfadjoint free random variables and $f$ 
  a bounded Borel function on the spectrum of
  $X$, then for $\abs{z}$ large enough, there exists a function $\delta$ such that
  \begin{equation*}
    \E_\varphi\left[\left.\left(z-X-f(X)Yf^*(X)\right)^{-1}\right|\mathcal{B}\right]=\left(z-X-\delta(z)f(X)f^*(X)\right)^{-1}.
  \end{equation*}
\end{proposition}
\begin{proof} 
  We start with a resolvent identity 
  \begin{multline*}
    \left(z-X-f(X)Yf^*(X)\right)^{-1}\\
    \begin{aligned}
    &=(z-X)^{-1}+\bigl(z-X- f(X)Yf^*(X)\bigr)^{-1}f(X)Yf^*(X) (z-X)^{-1}\\
    &=(z-X)^{-1}+(z-X)^{-1} \left(1- f(X)Y
      f^*(X)(z-X)^{-1}\right)^{-1}f(X)Yf^*(X) (z-X)^{-1}\\
    &=(z-X)^{-1}+(z-X)^{-1}f(X) \left(1-Y f^*(X)(z-X)^{-1}f(X)\right)^{-1}Yf^*(X) (z-X)^{-1}
    \end{aligned}
  \end{multline*}
  where we used the identity
  \[f(X) \left(1-Y f^*(X)(z-X)^{-1}f(X)\right)^{-1}= \left(1-f(X)Y f^*(X)(z-X)^{-1}\right)^{-1}f(X),\]
  which is immediate to verify.
  Now we apply the conditional expectation and obtain
  \begin{align*}
    \E_\varphi&\left[\left.\left(z-X-f(X)Yf^*(X)\right)^{-1}\right|\mathcal{B}\right]\\
    &=(z-X)^{-1}+(z-X)^{-1}f(X) E_\varphi\left[\left.\left(1-Y f^*(X)(z-X)^{-1}f(X)\right)^{-1}Y\right|\mathcal{B}\right]f^*(X) (z-X)^{-1}.
  \end{align*}
  In this form Lemma~\ref{lem:31} is applicable and yields
  \begin{align*}
    \E_\varphi&\left[\left.\left(z-X-f(X)Yf^*(X)\right)^{-1}\right|\mathcal{B}\right]\\
      &=(z-X)^{-1}+(z-X)^{-1}f(X) \delta(z)\left(1-\delta(z)f(X)(z-X)^{-1}f^*(X)\right)^{-1}f^*(X) (z-X)^{-1}\\
      &=(z-X-\delta(z)f(X)f^*(X))^{-1},
  \end{align*}
  where $\delta$ is the function given by
  \begin{equation}
    \label{eq:deltaseries}
  \begin{aligned}
    \delta(z)&=\eta_{f^*(X)(z-X)^{-1}f(X)}^Y(1)\\
             &=\sum_{n=0}^{\infty}\beta_{2n+1}(Y,f^*(X)(z-X)^{-1}f(X),Y,\ldots,f^*(X)(z-X)^{-1}f(X),Y)
  \end{aligned}
  \end{equation}
\end{proof}
Although the next proof follows closely the ideas of \cite[Section~3.3]{Biane98}
we include the details here for the reader's convenience.
First recall that the spectrum of any Hilbert space operator
$X$ is contained in the closure of its numerical range   \cite[Problem~214]{Halmos:hilbert2}
$$
W(X) = \{\langle X\xi,\xi\rangle \mid \xi\in\alg{H}, \norm{\xi}=1\}
$$
and that for a normal element operator the closure of the numerical range is actually equal to  the convex hull of the spectrum
\cite[Problem~216]{Halmos:hilbert2}.
\begin{lemma}
  \label{lem:convexhull}
  Let $\alg{A}$ be a finite von Neumann algebra, $\alg{B}\subseteq\alg{A}$ a von
  Neumann subalgebra and $E$ the conditional expectation.
  Then for any normal element $X\in\alg{A}$ the spectrum of $E[X|\alg{B}]$
  is contained in the closed convex hull of the spectrum of $X$.
\end{lemma}
\begin{proof}
  By the remarks preceding the lemma it 
  suffices to show that the numerical range of $E[X|\alg{B}]$
  is contained in the convex hull of the spectrum of $X$.
  Let $P(z)$ be the spectral resolution of $X$, then we can write
  \begin{align*}
     E[X|\alg{B}] = \int_{\sigma(X)} z E[dP(z)| \alg{B}]
  \end{align*}
  and thus
  \begin{align*}
    \langle E[X|\alg{B}]\xi,\xi\rangle
    &=    \int_{\sigma(X)} z \langle  E[dP(z)| \alg{B}]\xi,\xi\rangle
      .
  \end{align*}
  Now the map  $\mu(B)= \langle E[P(B)|\alg{B}]\xi,\xi\rangle$ is a
  probability measure ($\sigma$-additivity follows from normality of $E$) and
  the claim follows.
\end{proof}
\begin{proposition} 
  Given self-adjoint free random variables $X,Y$ and a bounded Borel function  $f$  on the
  spectrum of $X$, the identity \eqref{eq:mainthm:condexp}
  holds for all $z\in\mathbb{C}\setminus \mathbb{R}$.
\end{proposition} 
\begin{proof}
    First observe that the resolvent $\left(z-X-f(X)Yf^*(X)\right)^{-1}$ is
  an   analytic function of $z$ in the domain $\mathbb{C}\setminus\mathbb{R}$
  and has an absolutely convergent series expansion in every point,
  see, e.g., \cite{TaylorLay:1980}   or \cite{DunfordSchwartz:1958:I}.
  Applying the contractive map  $\E\left[\cdot| \alg{B}\right]$ to the series
  expansion  decreases the norm of the operator coefficients and  thus
  the function
  $h(z)=\E\left[\left.\left(z-X-f(X)Yf^*(X)\right)^{-1}\right|\alg{B}\right]$ is 
  analytic in $z$  for $z\in\mathbb{C}\setminus\mathbb{R}$ as well.
  Corollary~\ref{cor:kaplansky} we infer that $h(z)$ is a normal
  operator commuting with $X$.
  Moreover  positivity of the conditional expectation implies
  that $h(z)^*=h(\overline{z})$ and we restrict the further discussion
  to the upper half plane.

  For $z\in\IC^+$ 
  the resolvent $\left(z-X-f(X)Yf^*(X)\right)^{-1}$ is a bounded normal
  operator and its spectrum is a compact subset of the open disk with
  diameter $(0,-\tfrac{i}{\Im(z)})$, i.e., the disk of radius
  $\frac{1}{2\Im z}$ centred at $-\frac{i}{2\Im z}$. 
  Now by Lemma~\ref{lem:convexhull} for any $Z$ the spectrum of
  $\E_\varphi[Z | \mathcal{B}]$ is contained in the convex hull of the spectrum
  of $Z$
  and consequently the spectrum of $h(z)$ is a compact subset
  of the disk with diameter $(0,-\tfrac{i}{\Im(z)})$
  bounded away from 0.
  It follows that $h(z)$ is invertible
  and its spectrum is contained in the half plane  above the line $y=\Im(z)$.
  Since it is normal, 
  we infer from the remarks preceding Lemma~\ref{lem:convexhull}
  that its numerical range is contained therein as well and thus
  \begin{equation}
    \label{eq:Imhz>Imz}
  \Im \langle h(z)^{-1}\xi,\xi\rangle > \Im z
  \end{equation}
  for any $\xi$ with $\norm{\xi}=1$.
  Now for $z$ large enough this inverse is $h^{-1}(z)=z-X-\delta(z)f(X)f^*(X)$ 
  and thus for any vector $\xi$ with $\norm{\xi}=1$ and $f(X)^*\xi\ne0$
  we have for large $z$
  \[
    \delta(z)=\frac{z\langle \xi,\xi\rangle - \langle X\xi,\xi\rangle -\langle
      h^{-1}(z)\xi,\xi\rangle }    {\norm{f^*(X)\xi}^2}
  \]
  From this it follows  that $\delta(z)$ has an analytic extenstion to
  $\IC^+$ and together with   \eqref{eq:Imhz>Imz}  
  we conclude that $\Im\delta(z)<0$ throughout.
  Define
  \begin{equation*}
    \rho(z) = h(z)^{-1} + X +\delta(z)f(X)f^*(X).
  \end{equation*}
  By Proposition~\ref{prop:EzXfXYfX*}
  for $\abs{z}$ large enough we have $\rho(z)=z I$ and since $\rho(z)$
  depends analytically on $z$, it follows by analytic continuation   that
  $\rho(z)=z I$ in the whole upper half plane.
\end{proof}
\begin{proposition}
  The function $\delta(z)$ grows sublinearly in any nontangential sector,
  i.e., 
  $$
  \nontang{\abs{z}\to\infty} \frac{\delta(z)}{z}= 0
  $$ 
\end{proposition}
\begin{proof}
  We have seen in equation \eqref{eq:deltaseries} that 
  \begin{align*}
    \delta(z)&=\eta_{f^*(X)(z-X)^{-1}f(X)}^Y(1)\\
    &=\sum_{n=0}^{\infty}\beta_{2n+1}(Y,f^*(X)(z-X)^{-1}f(X),Y,\ldots,f^*(X)(z-X)^{-1}f(X),Y)\\
    &= \varphi(z) + \alg{O}(1/z).
  \end{align*}
\end{proof}
\begin{proposition}
    Assume that $f$ is not constantly zero on the spectrum of $X$. 
  Then the function $\delta(z)$ is uniquely determined by the properties
  \eqref{it:mainthm:delta} and   \eqref{it:mainthm:deltaequation1}
  of Theorem \ref{thm:41} for $\abs{z}$ large enough.
\end{proposition}
\begin{proof}
  From \eqref{eq:deltaseries} and Corollary~\ref{cor:multsubord} it follows 
  that in some neighbourhood of infinity $\delta(z)$ satisfies the equation
  \begin{equation}
    \label{eq:deltaequation}
    \mgfpsi_{f(X)(z-X)^{-1}f^*(X)Y}(1)=\mgfpsi_{f(X)(z-X)^{-1}f^*(X)}(\delta(z)).
  \end{equation}
  By analytic continuation we conclude that this equation holds for $z\in\C\setminus\R$. We will show that for $|z|$ large enough the above equation determines $\delta(z)$ uniquely.
  From \cite[Proposition 3.2]{Haagerup97} 
  we infer that for any $T$ (not necessarily self-adjoint) 
  with expectation $\varphi(T)\neq 0$
  the moment generating function $\mgfpsi_T$ is injective on the open disk
  \begin{align}\label{eq:deltaDom}
    \Bigl\{z:\abs{z}<\tfrac{\abs{\varphi(T)}}{4\,\norm{T}^2}\Bigr\}
  \end{align} 
  and the image contains the open disk
  \begin{align}\label{eq:deltaCodom}
    \Bigl\{z: \abs{z} < \tfrac{\abs{\varphi(T)}^2}{6\,\norm{T}^2}\Bigr\}.
  \end{align}
  We  will use this result for the operator $T=F(z-X)^{-1}F^*$,
  where here and   in the following we abbreviate $F=f(X)$,
  and show that equation    \eqref{eq:deltaequation}
  uniquely determines $\delta(z)$  for large $z$.
  In order to do so we shall  prove that the lower bounds
  \begin{equation}
    \label{eq:phiFzXFlowerbounds}
  \frac{\abs{\varphi(F(z-X)^{-1}F^*)}^2}{\norm{F(z-X)^{-1}F^*}^2}
  \geq \frac{\varphi(FF^*)^2}{4\norm{F}^4}
  \quad
  \text{ and }
  \quad
  \frac{\abs{\varphi(F(z-X)^{-1}F^*)}}{\abs{z}\,\norm{F(z-X)^{-1}F^*}^2}
  \geq \frac{\varphi(FF^*)}{4\norm{F}^4}
  \end{equation}
  hold uniformly for 
  $$
  \abs{z}> 2\norm{X}
  \Bigl(
    1+\frac{\norm{F}^2}{\varphi(FF^*)}
  \Bigr)
  .
  $$
    To conclude, we argue that
  \begin{enumerate}
   \item the values   $\mgfpsi_{f(X)(z-X)^{-1}f^*(X)Y}(1)$ are well defined and converge to $0$ as $\abs{z}$ 
    goes to infinity and therefore they are contained in the ball \eqref{eq:deltaCodom} for large $\abs{z}$.
   \item
    The quotient that $\delta(z)/z\to 0$ as $z\to\infty$ and therefore
    the values $\delta(z)$ are contained in the ball \eqref{eq:deltaDom} for large $\abs{z}$.
  \end{enumerate}

  Let us first estimate the quotient $\frac{\abs{\varphi(T)}}{\norm{T}}$ from below.
  To this end we expand the numerator for $|z|>\norm{X}$ into a Neumann series
  \begin{align*}
    \abs{
      \varphi\bigl(
        (z-X)^{-1}FF^*
         \bigr)
     }
   &=\frac{1}{\abs{z}}
     \,\abs{\varphi\left(
         \left(
           1-z^{-1}X
         \right)^{-1}FF^*
       \right)
     }
   \\
    &=
    \frac{1}{\abs{z}}
    \biggabs{
      \sum_{n=0}^{\infty}
       \frac{1}{z^n}
       \varphi\left(X^nFF^*\right)
    }\\
    &\geq \frac{1}{\abs{z}}
          \biggl(
            \varphi(FF^*)
            -\sum_{n=1}^{\infty}\frac{1}{\abs{z}^n}\norm{X}^n\norm{F}^2
          \biggr)
   \\ 
   &= \frac{1}{\abs{z}}
      \left(
        \varphi(FF^*)
        -\frac{\norm{X}\norm{F}^2}{\abs{z}-\norm{X}}
      \right).
    \end{align*}
    On the other hand for $\abs{z}>\norm{X}$ we have
    \begin{equation}
      \label{eq:normz-X1FF}
      \norm{(z-X)^{-1}FF^*}
      = \frac{1}{\abs{z}}
        \norm{
          \left(
            1-z^{-1}X
          \right)^{-1}FF^*}
     \leq \frac{1}{\abs{z}}
          \sum_{n=0}^{\infty}\frac{\norm{X}^n}{\abs{z}^n}\norm{F}^2
     =\frac{\norm{F}^2}{\abs{z}-\norm{X}}.
   \end{equation}
   Thus for $\abs{z}>\norm{X}$ we can estimate
   \begin{align*}
          \frac{
       \bigabs{
         \varphi\bigl((z-X)^{-1}FF^*\bigr)
       }}     {\norm{(z-X)^{-1}FF^*}}
     \geq \frac{1}{\abs{z}}
          \left(
            \frac{
            \varphi(FF^*)
            -\frac{\norm{X}\,\norm{F}^2}{\abs{z}-\norm{X}}}                 {\frac{\norm{F}^2}{\abs{z}-\norm{X}}}
          \right)
     =\frac{1}{\norm{F}^2}
      \left(
        \varphi(FF^*)
        -\frac{\norm{X}(\varphi(FF^*)+\norm{F}^2)}{\abs{z}}
      \right)
    \end{align*}
    and consequently for
    $\abs{z}>2 \norm{X}(\varphi(FF^*)+\norm{F}^2)\,\varphi(FF^*)^{-1}$ 
    we have the lower bound
    \begin{equation*}
      \frac{\abs{\varphi\bigl(\left(z-X\right)^{-1}FF^*\bigr)}}           {\norm{(z-X)^{-1}FF^*}}
           \geq\frac{\varphi(FF^*)}{2\norm{F}^2}.
    \end{equation*}

    The first bound in \eqref{eq:phiFzXFlowerbounds} is an immediate consequence of this estimate.
    On the other hand, for $\abs{z}>2 \norm{X}(\varphi(FF^*)+\norm{F}^2)\,\varphi(FF^*)^{-1}$ 
    we can use \eqref{eq:normz-X1FF} to estimate
    \begin{align*}
      \abs{z}\,      \norm{(z-X)^{-1}FF^*}
      &\leq \frac{\norm{F}^2}{1-\frac{\norm{X}}{\abs{z}}}\\
      &\leq \frac{\norm{F}^2}{1-\frac{\varphi(FF^*)}{2(\varphi(FF^*)+\norm{F}^2)}}\\
      &= \frac{2\norm{F}^2(\varphi(FF^*)+\norm{F}^2)}{\varphi(FF^*)+2\norm{F}^2)}\\
      &\leq 2\norm{F}^2
    \end{align*}
    and the second bound in \eqref{eq:phiFzXFlowerbounds} follows.

\end{proof}

\begin{proposition}
  The function $\delta(z)$ is uniquely determined by properties
  \eqref{it:mainthm:delta} and   \eqref{it:mainthm:deltaequation2}
  of Theorem \ref{thm:41}.
\end{proposition}
\begin{proof}
  We have
  \begin{equation*}
    \mgfpsi_{f(X)(z-X)^{-1}f^*(X)Y}(1)=\mgfpsi_{f(X)(z-X)^{-1}f^*(X)}(\delta(z)).
  \end{equation*}
  Since $\eta_U(z)=\mgfpsi_U(z)/(1+\mgfpsi_U(z))$ we have equivalently
  \begin{equation*}
    \eta_{f(X)(z-X)^{-1}f^*(X)Y}(1)=\eta_{f(X)(z-X)^{-1}f^*(X)}(\delta(z)).
  \end{equation*}
  Thus one gets
  \begin{equation*}
    1=\eta^{\langle-1\rangle}_{f(X)(z-X)^{-1}f^*(X)Y}\left(\eta_{f(X)(z-X)^{-1}f^*(X)}(\delta(z))\right).
  \end{equation*}
  Recall from   \eqref{eq:sigmamuboxtimesnu} that for $U,V$ free 
  we have
  $\eta^{\langle-1\rangle}_{UV}=\tfrac{\eta^{\langle-1\rangle}_{U}\eta^{\langle-1\rangle}_{V}}{z}$ and thus
  \begin{equation*}
    \eta_{f(X)(z-X)^{-1}f^*(X)}(\delta(z))=\delta(z)\eta^{\langle-1\rangle}_{Y}\left(\eta_{f(X)(z-X)^{-1}f^*(X)}(\delta(z))\right).
  \end{equation*}
  So finally we obtain the desired fixed point equation
  \begin{equation*}
    \frac{\eta_{Y}\left(\frac{\eta_{f(X)(z-X)^{-1}f^*(X)}(\delta(z))}{\delta(z)}\right)}{\frac{\eta_{f(X)(z-X)^{-1}f^*(X)}(\delta(z))}{\delta(z)}}=\delta(z)
    .
  \end{equation*}
  It is easy to see that $\delta$ satisfies
  \eqref{eq:mainthm:deltaequation1}      
  of Theorem if and only if it satisfies
  \eqref{eq:mainthm:deltaequation2} and this finishes the proof.
\end{proof}
\section{The distribution of  $X+f(X)Yf^*(X)$ and its relation to free convolutions}
\label{sec:examples}

\subsection{An algorithm}
The results of the previous section yields the following
effective method of calculating explicitly the distribution of $X+f(X)Yf^*(X)$
for any Borel function $f$.
\begin{theorem}\label{Cor:CantConv}
Let $X$ and $Y$ be free selfadjoint random variables. 
Then the following procedure yields the Cauchy transform of  $X+f(X)Yf^*(X)$.
\begin{enumerate}[1.]
 \item
  Calculate the moment generating function of $f(X)(z-X)^{-1}f(X)^*$
  \begin{equation*}
    \mgfpsi_{f(X)(z-X)^{-1}f(X)^*}(s) = \int \frac{sx}{1-\frac{s\,\abs{f(x)}^2}{z-x}}\,d\mu(x)
  \end{equation*}
 \item 
  Calculate the shifted Boolean cumulant generating functions
  $\widetilde{\eta}_Y(s)$ and $\widetilde{\eta}_{f(X)(z-X)^{-1}f^*(X)}(s)$ according to 
  \eqref{eq:scgf}.
 \item 
  Solve the fixed point equation \eqref{eq:mainthm:deltaequation2} or the
  equivalent equation
  \begin{equation}
     \label{eq:algorithm:deltaequation2}
   \widetilde{\eta}_{f(X)(z-X)^{-1}f(X)^*}(\delta(z))=  \widetilde{\eta}_Y^{-1}(\delta(z)).
  \end{equation}
 \item 
  Evaluate the integral
  \begin{equation}\label{eq:CauchyTrCantConv}
    G_{X+f(X)Yf^*(X)}(z)=\int_{\mathbb{R}}\frac{1}{z-x-\delta(z)\abs{f(x)}^2}d\mu(x),
  \end{equation}
  where $\delta(z)$ is determined as it is described in Theorem \ref{thm:41}.
\end{enumerate}
\end{theorem}
\begin{proof}
  The steps of the algorithm can be extracted directly from Theorem~\ref{thm:41}.
  The final equation follows by applying $\varphi$ to both sides of \eqref{it:mainthm:condexp}.
\end{proof}

In the remainder of this section we discuss the  relation of our procedure to previously
known methods of calculating free convolutions: In the case
$f\equiv 1$ one recovers additive convolution and for positive $X$ and
$f(x)=\sqrt{x}$ the operation which we consider is equivalent to
the multiplicative free convolution of $X$ and $Y+1$.
We show that in the additive case the fixed point equation from
Theorem~\ref{thm:41} is equivalent to the fixed point equations found in
\cite{BelBerNewApproach}. We finish this section with examples of explicit
calculations using our method. 
In the examples computed below and in Section~\ref{sec:algebraic} it will turn
out that in practice equation
\eqref{eq:algorithm:deltaequation2} is more accessible than
equation \eqref{eq:mainthm:deltaequation2}.
     
\subsection{Fixed point equations for additive free convolution}
We return to the question of subordination for additive free convolution
and show how it can be derived as a special case.
In addition we show how the fixed point equation \cite[Theorem~4.1]{BelBerNewApproach} can be derived
as a consequence of the fixed point equation \eqref{eq:mainthm:deltaequation2}.

\begin{example}\label{ex:addConv}
  Given free random variables $X$ and $Y$
  we observed in Corollary~\ref{cor:additivesubordination}
  that
  \[G_{X+Y}(z)=G_X\left(z-\eta^{Y}_{(z-X)^{-1}}(1)\right),\]
  where
  \begin{align*}
    \eta^{Y}_{(z-X)^{-1}}(s)=\sum_{n=0}^{\infty}\beta_{2n+1}(Y,(z-X)^{-1},Y,\ldots,(z-X)^{-1},Y)s^{2n+1}.
  \end{align*}
  This implies that one can find the free additive
  convolution of $X$ and $Y$, by calculating the free multiplicative
  convolution of $Y$ and $(z-X)^{-1}$ for $z$ in some neighbourhood of
  infinity. More precisely one can find the function
  $\eta^{Y}_{(z-X)^{-1}}(1)$ as a subordination function of free
  multiplicative convolution of $Y$ and $(z-X)^{-1}$.
  
  Next we will use equation
  \eqref{eq:mainthm:deltaequation2}      
  from Theorem~\ref{thm:41} with $f\equiv 1$ to derive a known fixed point equation for subordination function of free additive convolution.
  Then $z-\delta(z)$ satisfies the fixed point equation of the subordination function.
    First we note that
    \begin{equation*}
      \widetilde{\eta}_{(z-X)^{-1}}(\delta(z))
      =\frac{\mgfpsi_{(z-X)^{-1}}(\delta(z))}{\delta(z)\left(1+\mgfpsi_{(z-X)^{-1}}(\delta(z))\right)}
    \end{equation*}
    and
    \begin{equation*}
      \mgfpsi_{(z-X)^{-1}}(s)=\int_{\mathbb{R}}\frac{s(z-x)^{-1}}{1-s(z-x)^{-1}}d\mu_X(x)=sG_X(z-s),
    \end{equation*}
    thus
    \begin{equation}\label{eq:etaT}
      \widetilde{\eta}_{(z-X)^{-1}}(\delta(z))=\frac{G_X(z-\delta(z))}{1+\delta(z)G_X(z-\delta(z))}.
    \end{equation}
    On the other hand
    \begin{equation*}
      \widetilde{\eta}_Y(s)=\frac{\mgfpsi_Y(s)}{s(1+\mgfpsi_Y(s))}=\frac{\frac{1}{s}G_Y\left(\frac{1}{s}\right)-1}{G_Y\left(\frac{1}{s}\right)}=\frac{G_Y\left(\frac{1}{s}\right)-s}{sG_Y\left(\frac{1}{s}\right)}
    \end{equation*}
    and hence the fixed point equation
    \begin{equation*}
      \delta(z)=\widetilde{\eta}_Y\left(\widetilde{\eta}_{(z-X)^{-1}}(\delta(z))\right)
    \end{equation*}
    becomes
    \begin{equation*}
      \delta(z)
     =\frac{G_Y\Bigl(\frac{1}{\widetilde{\eta}_{(z-X)^{-1}}(\delta(z))}\Bigr)-\widetilde{\eta}_{(z-X)^{-1}}(\delta(z))}           {\widetilde{\eta}_{(z-X)^{-1}}(\delta(z))\,G_Y\Bigl(\frac{1}{\widetilde{\eta}_{(z-X)^{-1}}(\delta(z))}\Bigr)}
     =\frac{1}{\widetilde{\eta}_{(z-X)^{-1}}(\delta(z))}
      -\frac{1}{G_Y\Bigl(\frac{1}{\widetilde{\eta}_{(z-X)^{-1}}(\delta(z))}\Bigr)}.
    \end{equation*}
    Finally using \eqref{eq:etaT} we get
    \begin{equation*}
      \delta(z)=\delta(z)+\frac{1}{G_X(z-\delta(z))}-\frac{1}{G_Y\Bigl(\delta(z)+\frac{1}{G_X(z-\delta(z))}\Bigr)}.
    \end{equation*}
    With $F_U(z)=\tfrac{1}{G_U(z)}$, we get
    \begin{equation*}
      F_X(z-\delta(z))=F_Y\left(\delta(z)+F_X(z-\delta(z))\right).
    \end{equation*}
    It is straightforward fo check that with $h_1(\lambda)=F_X(\lambda)-\lambda$, $h_2(\lambda)=F_Y(\lambda)-\lambda$ and $\omega(z)=z-\delta(z)$  the above equation is equivalent to
    \begin{equation*}
      \omega(z)=z+h_2(z+h_1(\omega(z)))
    \end{equation*}
    which is exactly the fixed point equation for subordination function of the
    additive free convolution found in the proof of
    \cite[Theorem~4.1]{BelBerNewApproach}.
  \end{example}

  \begin{remark}
    In a similar way the fixed point equation for the subordination function for the free multiplicative convolution can be derived. 
  \end{remark}

    \subsection{Example: Bernoulli laws}
  Next we determine with our method the free convolution of the Bernoulli law $1/2\left(\delta_{-1}+\delta_1\right)$ with itself.   \begin{example} \label{eg:35}
    Let $X,Y$ be free both distributed $1/2\left(\delta_{-1}+\delta_1\right)$.
    The relevant transforms are
    \begin{align}\label{eq:eg1}
      G_X(z)&=G_Y(z)=\frac{z}{z^2-1}, & \eta_X(z)&=\eta_Y(z)=z^2, & \eta_{(z-X)^{-1}}(s)&=\frac{s(s-z)}{1+z(s-z)}
    \end{align}
    From Example~\ref{ex:addConv} it follows that
    \begin{equation*}
      G_{X+Y}(z)=G_X\left(z-\eta^{Y}_{(z-X)^{-1}}(1)\right),
    \end{equation*} 
    Thus in order to calculate the distribution of $X+Y$ all we need to find is $\eta^{Y}_{(z-X)^{-1}}(1)$.
    On the other hand from \eqref{eq:MSTz} we have the identity
    \begin{equation*}
      \mgfpsi_{(z-X)^{-1}Y}(1)=\mgfpsi_{(z-X)^{-1}}\left(\eta^{Y}_{(z-X)^{-1}}(1)\right)
    \end{equation*}
    which is equivalent to
    \begin{equation*}
      \eta_{(z-X)^{-1}Y}(1)=\eta_{(z-X)^{-1}}\left(\eta^{Y}_{(z-X)^{-1}}(1)\right),
    \end{equation*}
    hence we have
    \begin{equation}\label{eq:eg2}
      \eta^{Y}_{(z-X)^{-1}}(1)=\eta^{\langle-1\rangle}_{(z-X)^{-1}}\left(	\eta_{(z-X)^{-1}Y}(1)\right).
    \end{equation}
    Since $(z-X)^{-1}$ and $Y$ are free we have 
    \begin{equation*}
      \eta^{\langle -1 \rangle}_{(z-X)^{-1}Y}(s)=\eta^{\langle -1 \rangle}_{(z-X)^{-1}}(s)\eta^{\langle -1 \rangle}_{Y}(s)/s,
    \end{equation*}
    substituting $s:=\eta_{(z-X)^{-1}}(s)$ we obtain
    \begin{equation*}
      \eta^{\langle -1 \rangle}_{(z-X)^{-1}Y}\left(\eta_{(z-X)^{-1}}(s)\right)=s\eta^{\langle -1 \rangle}_{Y}\left(\eta_{(z-X)^{-1}}(s)\right)/\eta_{(z-X)^{-1}}(s),
    \end{equation*}
    Observe that by \eqref{eq:eg2} we have that
    $\eta^{Y}_{(z-X)^{-1}}(1)$ is the inverse at $1$ of the LHS of the above equation. Denoting $\lambda_z=\eta^{Y}_{(z-X)^{-1}}(1)$ after simple transformations we get
    \begin{align*}
      \eta_Y\left(\frac{\eta_{(z-X)^{-1}}\left(\lambda_z\right)}{\lambda_z}\right)=\eta_{(z-X)^{-1}}(\lambda_z).
    \end{align*}
    Using \eqref{eq:eg1} we get
    \begin{equation*}
      \left(\frac{\lambda_z (\lambda_z - z)}{1 + (\lambda_z - z) z}\right)^2=\frac{(\lambda_z - z)}{-1 + (\lambda_z - z) z}
    \end{equation*}
    From the relation $G_{X+Y}(z)=G_X(z-\lambda_z)$ we see that $\lambda_z\neq z$ and we can divide both sides by $(\lambda_z-z)/(1+(\lambda_z-z)z)$. Solving the resulting equation we get
    \[\lambda_z=\frac{z\pm\sqrt{z^2-4}}{2},\]
    and we choose the branch for which $\lambda_z/z\to 0$ as $z\to\infty$, thus $\lambda_z=\tfrac{z-\sqrt{z^2-4}}{2}$.
    Finally after a simple calculation we get
    \[G_{X+Y}(z)=G_X(z-\lambda_z)=\frac{1}{\sqrt{z^2-4}}.\]
    Thus we recover the well known fact that $X+Y$ has the arcsine distribution.
  \end{example}

  \subsection{Example: Free compression}
  For the next example we take free random variables $P,Y$ where $P$ is a projection and $Y$ has semicircular distribution and calculate the distribution of $P+PYP$. Since the distribution of $P+PYP$ is the same as the distribution of $P(1+Y)P$ one can find it using $S$-transform, so again the result is not new.
  \begin{example}
    Let $P$ be a projection such that $\varphi(P)=p>0$ and assume that $Y$ is free from $P$ and $Y$ has Wigner semicircle law.
    We will calculate the distribution of $P+PYP$. From Theorem \ref{Cor:CantConv} we get that 
    \begin{equation}\label{eq:GprojSemi}
      G_{P+PYP}(z)=\int_{\mathbb{R}}\frac{1}{z-x-\delta(z)x}d\mu_P(x)=\frac{1-p}{z}+\frac{p}{z-1-\delta(z)},
    \end{equation}
    where the function $\delta$ satisfies
    \begin{equation}\label{eq:deltaprojSemi}
      \widetilde{\eta}_{Y}\left(\widetilde{\eta}_{P(z-P)^{-1}P}(\delta(z))\right)=\delta(z).
    \end{equation}
    
    The Cauchy transform of the semicircle law satisfies the quadratic equation
    $G_Y(z)^2-zG_Y(z)+1=0$ and a simple calculation shows that the shifted
    Boolean cumulant generating function satisfies the quadratic equation
    $$
    z\widetilde{\eta}(z)^2-\widetilde{\eta}+z=0
    $$
    or equivalently
    \begin{equation}
      \label{eq:etatildesemicircleimplicit}      
      \frac{\widetilde{\eta}_Y(z)}{1+\widetilde{\eta}_Y(z)^2} = z
      .
    \end{equation}
    On the other hand
    \begin{equation*}
      \widetilde{\eta}_{P(z-P)^{-1}P}(s)=\frac{p}{z-1-s(1-p)}.      
    \end{equation*}
    and together with  \eqref{eq:deltaprojSemi} we get the equation
    $$
    \frac{\delta(z)}{1+\delta(z)^2}= \frac{p}{z-1-\delta(z)(1-p)}
    .
    $$
    The solution is   
    \begin{equation*}
      \delta(z)=\frac{z-1-\sqrt{(z-1)^2-4p}}{2},
    \end{equation*}
    and substituting this into \eqref{eq:GprojSemi} we finally get
    \begin{equation*}
      G_{P+PYP}(z)=\frac{1-p}{z}+\frac{z-1-\sqrt{(z-1)^2-4p}}{2}.
    \end{equation*}
  \end{example}
  
\section{Algebraic Cauchy transforms}
\label{sec:algebraic}

In this section we study in detail an example which cannot be easily treated with previous tools from free harmonic analysis. We assign to $X$ and $Y$  semicircle laws and consider the simplest nontrivial function $f(x)=x$.
This example can serve as a template for an algorithm which allows in principle to determine an explicit equation for the Cauchy transform of $X+p(X)Yp^*(X)$ for any polynomial $p$ and for any pair of free random variables $X,Y$ with algebraic Cauchy transforms.
The algorithm consists in a careful application
of classical tools from algebraic geometry.
We evaluate the Cauchy integrals  implicitly and
then eliminate auxiliary variables  via resultants.
In the final step the correct equation is selected using 
Newton polygons
\cite[\S{}8.3]{BrieskornKnorrer:1986:plane}.

\subsection{Resultants and Elimination}
A basic tool for analyzing algebraic equations consists in elimination, that
is, computing projections of algebraic varieties to lower dimensions. There are
basically two algorithmic approaches to achieve this: Gr\"obner bases and
resultants. We stick to resultants here because they are well suited for small
systems like ours and allow for manageable step-by-step computations.

The basic problem of elimination can be boiled down to an iteration of the following question:
Given two polynomials $f(x)$ and $g(x)$, do they have a common zero?
The answer to this question can be read off a certain polynomial in the
coefficients of $f$ and $g$ called  the \emph{resultant},
which can be obtained by a fraction free version of the Euclidean algorithm. 
\begin{theorem}[{\cite[\S3.6]{CLO:2015:ideals},
    \cite[Chapter~3]{CLO:2005:using},\cite{Sturmfels:2002:solving}}]
  Given polynomials $f(x)=a_0+a_1x+\dots+a_mx^m$ and
  $g(x)=b_0+b_1x+\dots+b_nx^n$
  we define their \emph{resultant} by
  \begin{equation*}
    \Res(f,g) = a_m^nb_n^m\prod_{i,j}(\xi_i-\eta_j)
    ,
  \end{equation*}
  where $\xi_1,\xi_2,\dots,\xi_m$ and $\eta_1,\eta_2,\dots,\eta_n$
  are the roots of $f(x)$ and $g(x)$,     respectively.
  The resultant has the following properties:
  \begin{enumerate}[1.]
   \item $\Res(f,g)$ is an integer polynomial in the coefficients $a_i$ and
    $b_j$.
   \item $\Res(f,g)=0$ if and only if $f$ and $g$ have a nontrivial common
    factor.
   \item There are polynomials $A(x)$ and $B(x)$ such that
    $A(x)f(x)+B(x)g(x) = \Res(f,g)$.
   \item $\Res(f_1f_2,g)=\Res(f_1,g)\Res(f_2,g)$.
   \item $\Res(f,g)=(-1)^{mn}\Res(g,f)$.
  \end{enumerate}
  There are various methods to compute the resultant:
  \begin{enumerate}[1.]
   \item It is the determinant of the Sylvester matrix
    \begin{equation*}
      \Res(f,g) =
      \begin{vmatrix}
        a_0   &      &      &      &b_0\\
        a_1   &a_0   &      &      &b_1   &b_0\\
        a_2   &a_1   &\ddots&      &b_2   &b_1   &\ddots\\
        \vdots&a_2   &\ddots&a_0   &\vdots&b_2   &\ddots&b_0\\
        a_m   &\vdots&\ddots&a_1   &b_n   &\vdots&\ddots&b_1\\
              &  a_m &      &a_2   &      &b_n   &      &b_2\\
              &      &\ddots&\vdots&      &     &\ddots&\vdots\\
              &      &      &a_m   &      &     &      &b_n
      \end{vmatrix}
    \end{equation*}
   \item
    Denote by $A_f=\IC[x]/\langle f\rangle$ the quotient ring and by
    $M_g:A_f\to A_f$     the multiplication operator $M_g[h]_{A_f}=[gh]_{A_f}$.
    Then
    \begin{equation*}
      \Res(f,g) = a_m^n\det M_g
      .
    \end{equation*}
    Equivalently, since the matrix $M_x$ of multiplication by $x$ with respect
    to the canonical basis of $A_f$ is given by the companion matrix
    $$
    C_f =
    \begin{bmatrix}
      0&0&\dots&0&-a_0/a_m\\
      1&0&\dots&0&-a_1/a_m\\
      0&1&\dots&0&-a_2/a_m\\
      \vdots&\vdots&\ddots&\vdots&\vdots\\
      0&0&\dots&1&-a_{m-1}/a_m
    \end{bmatrix}
    $$
    the resultant is
    \begin{equation*}
      \Res(f,g) = a_m^n\det g(C_f)
    \end{equation*}
   \item
    If $f$ and $g$ have the same degree $m=n$ then with
    $$
    h(x,y) = \frac{f(x)g(y)-f(y)g(x)}{x-y} = \sum_{i,j=0}^{n-1} c_{ij}x^iy^j
    $$
    the resultant is equal to the \emph{B\'ezout determinant}
    \begin{equation*}
      \Res(f,g) = \det [c_{ij}]_{i,j=0,1,\dots,n-1}
    \end{equation*}
   \item
    Assume $m\geq n$, then
    the resultant can be computed by a variant of the Euclidean algorithm:
    If $f(x)=q(x)g(x)+r(x)$, then
    \begin{equation*}
      \Res(f,g)=(-1)^{mn}b_n^{m-\deg r(x)}\Res(g,r)
    \end{equation*}
  \end{enumerate}
\end{theorem}
When applied to multivariate polynomials, i.e., when the coefficients of the
polynomials are polynomials in other variables themselves, resultants can be
used to effectively eliminate variables from systems of polynomial equations.
In this case we will write $\Res(f,g,x)$ to emphasize which variable is to be
eliminated. 

Calculation of the resultant by hand is possible, but tedious, in
particular when it is done iteratively then the rapid growth of the
resulting expressions makes it infeasible after a few steps. Therefore
these calculations are best left to a computer algebra system.

\subsection{Arithmetics of algebraic numbers}
Most of the equations we will encounter have no explicit solution and
we will work with implicit equations exclusively. The arithmetics are
similar to those of the field of algebraic numbers and can be done
implicitly.  More precisely, given two algebraic numbers $\alpha$ and
$\beta$ and polynomials $f$ and $g$ such that $f(\alpha)=0$ and
$g(\beta)=0$, the following table provides polynomial equations for
$\alpha\pm\beta$,
$\alpha\cdot \beta$ and $1/\alpha$:
\begin{subequations}
  \begin{align}
    \alpha+\beta &: \phantom=\Res(f(x-y),g(y),y)\\
    \alpha-\beta &: \phantom=\Res(f(x+y),g(y),y) \label{eq:alpha-beta}\\
    \alpha\cdot\beta&: \phantom=\Res(f(xy),\tilde{g}(y),y)\\
    \alpha/\beta&: \phantom=\Res(f(xy),g(y),y)  \label{eq:alpha/beta}
  \end{align}
\end{subequations}
where $\tilde{g}(y)=g(1/y)y^n$ is the reverse polynomial of a polynomial $g$
of degree $n$. Its roots are the reciprocals of the roots of $g$.

\subsection{Partial fractions and Cauchy transforms}
We will encounter  integrals of the form
$$
\int \frac{p(x)}{q(x)}\,d\mu(x)
$$
with rational integrands which can be evaluated in terms of the Cauchy
transform $G_\mu(z)$. Indeed, using divided differences we can perform a
partial fraction expansion and obtain for example
\begin{equation}
  \label{eq:divdiff}
  \int \frac{\alpha + \beta x }{(\lambda_1-x)(\lambda_2-x)}\,d\mu(x)
  =-\frac{(\alpha+\beta\lambda_1)G_\mu(\lambda_1)-(\alpha+\beta\lambda_2)G_\mu(\lambda_2)}{\lambda_1-\lambda_2}
\end{equation}

\subsection{Puiseux expansions}
\label{ssec:puiseux}
Resultants often yield reducible equations and in order to keep the
calculations manageable it is advisable to eliminate inessential factors at
every step. In the calculations below this turns out to be easy with the
exception of the final equation, which has three nontrivial factors.
In order to sort out the correct branch it is necessary to study the
Puiseux expansions of the solutions \cite[\S{}4.3]{Walker:1950:algebraic}.
Let $P(x,y)$ be a bivariate polynomial over $\IC$ and denote its degree $y$ by $n$.
Then the fundamental theorem of algebra implies that for every $x$ where the
leading
coefficient does not vanish the equation $P(x,y)=0$ has $n$ solutions (counting
multiplicity) and if in addition at least one partial derivative of $P(x,y)$
does not vanish then we can infer 
from the implicit function theorem that the solutions $y(x)$
are holomorphic and can be expanded into a convergent Taylor series.

However at singular points, where either the leading coefficient 
or both partial derivatives vanish this is not true anymore because the field
$\IC(x)$
of Taylor series is not algebraically closed.
This limitation can be overcome by using fractional power series or
Newton-Puiseux series
\begin{equation}
  \label{eq:puiseux}
y(x) = \sum_{k=k_0}^\infty c_k x^{k/n}  
\end{equation}
where $n\in\IN$ and $k_0\in\IZ$.

The field $\IC(z)^*$ of Newton-Puiseux series is algebraically closed
(see 
\cite[Corollary~13.15]{Eisenbud:1995:commutative} or
\cite[Theorem~4.3.1]{Walker:1950:algebraic}) and the solutions can
be found using an iterative process using \emph{Newton polygons},
see
\cite[\S{}4.3.2]{Walker:1950:algebraic} or
\cite[\S{}8.3]{BrieskornKnorrer:1986:plane} for an extensive discussion
including Newton's original letter which describes the method still in use
today.
Let $P(x,y)=\sum_{k=0}^np_k(x)\,y^k$
and denote by $\alpha_k$ the lowest power of $x$ appearing in the $k$-th
coefficient
$p_k(x)$, i.e., $p_k(x)=x^{\alpha_k}(a_k+b_k(x))$.
Knowing that a solution of the form \eqref{eq:puiseux} exists we make the
ansatz
$$
y(x) = x^\gamma(\eta+y_1(x))
$$
where $y_1(x)$ is a series in $x$ with positive exponents.
Collecting the lowest order contribution of each term of $P(x,y(x))$
we see that
$$
\sum_{k=0}^n (a_kx^{\alpha_k}+b_k(x))\,x^{k\gamma}(\eta+y_1(x))^k
=\sum_{k=0}^n a_kx^{\alpha_k+k\gamma}\eta^k
+P_1(x,\eta,y_1(x)) = 0
$$
where the powers of $x$ appearing in $P_1(x,\eta,y_1(x))$ are strictly larger
then those in the \emph{critical polynomial}
$$
\sum_{k=0}^n a_kx^{\alpha_k+k\gamma}\eta^k
.
$$
Let $\beta = \min\{\alpha_k+\gamma k\mid k \in 0,1,\dots,n\}$ be the
smallest exponent appearing in the critical polynomial.  In order for
a solution to exist, the lowest order terms must cancel and therefore
the power $\beta$ must appear at least twice.  This condition is
visualized in the \emph{Newton polygon} which is the convex hull of
the points $\{(k,\alpha_k)\mid k=1,2,\dots,n\}$ in the plane and the
admissible values of $\gamma$ can be read off the slopes of the lines
at the lower boundary of this convex polygon.

\subsection{An equation for the Cauchy transform of $T=X(z-X)^{-1}X$}
\label{ssec:CauchyT}
We first compute the Cauchy transform of $T=X(z-X)^{-1}X$.
To this end in the present section we denote by $\mu$ the standard
semicircle law.
Its  Cauchy transform $G_\mu(z)$ 
is the inverse of the Zhukovsky
transform and satisfies the equation
$G_\mu(z)+1/G_\mu(z)=z$, i.e.,
\begin{equation}
  \label{eq:WignerCauchy}
  G_\mu(z)^2-zG_\mu(z)+1=0
  .
\end{equation}
The Cauchy transform of $T$ is the integral
$$
G_T(s) = \int \frac{d\mu(x)}{s-\frac{x^2}{z-x}}
= \int \frac{z-x}{sz-x^2-sx}\, d\mu(x)
$$
which evaluates via \eqref{eq:divdiff}  to
\begin{equation}
  \label{eq:CauchyGT}
G_T(s) =
-\frac{(\lambda_1-z)G_\mu(\lambda_1)-(\lambda_2-z)G_\mu(\lambda_2)}{\lambda_1-\lambda_2}
=
\frac{1}{s}\frac{\lambda_1^2G_\mu(\lambda_1)-\lambda_2^2G_\mu(\lambda_2)}{\lambda_1-\lambda_2}
\end{equation}
where $\lambda_1$ and $\lambda_2$ are the roots of the polynomial
\begin{equation}
  \label{eq:semisemi:denom}
  \lambda^2+s\lambda-sz = 0
\end{equation}
or equivalently,
\begin{align*}
  \lambda_1+\lambda_2 &= -s\\
  \lambda_1\lambda_1&=-sz
  .
\end{align*}
We first compute an equation for the denominator of \eqref{eq:CauchyGT}.
In order to do this we compute the resultant of type \eqref{eq:alpha-beta}
of the equation  \eqref{eq:semisemi:denom} which is satisfied by both $\lambda_i$
and obtain
$$
p_{\Delta\lambda}(x)= \Res(  (x+y)^2+s(x+y)-sz , y^2+sy-sz,y)
= x^2(x^2-4sz-s^2)
.
$$
Next we compute an equation for the numerator of  \eqref{eq:CauchyGT}.
From  \eqref{eq:WignerCauchy} we infer that the
function $\widetilde{G}_\mu(\lambda)=\lambda^2G_\mu(\lambda)$
is algebraic and satisfies the equation
$$
\frac{\lambda^2}{\widetilde{G}_\mu(\lambda)} + \frac{\widetilde{G}_\mu(\lambda)}{\lambda^2} = \lambda
$$
that is, it is  a root of the polynomial
$$
x^2-\lambda^3x+\lambda^4=0
.
$$
Taken together with   \eqref{eq:semisemi:denom} we see that
$\widetilde{G}_i=\lambda_i^2G(\lambda_i)$ is the $y$-component of the solution of the
algebraic system
\begin{equation*}
  \begin{aligned}
    \lambda^2+s\lambda-sz &= 0\\
    \widetilde{G}^2-\lambda^3\widetilde{G}+\lambda^4&=0
  \end{aligned}
\end{equation*}
We can thus obtain an equation for $\widetilde{G}_i$ as a function of $s$ and $z$ by
computing the resultant with respect to $\lambda$:
\begin{align*}
  p_{\widetilde{G}}(x)
    &= \Res(    \lambda^2+s\lambda-sz,    x^2-\lambda^3x+\lambda^4,
    \lambda)\\
    &=
  x^4
+( 3  s^2  z+s^3 ) x^3
+(  -s^3  z^3+2  s^2  z^2+4  s^3  z+s^4 ) x^2
-s^4  z^3  x
+s^4  z^4 
\end{align*}
The equation for the difference $\widetilde{G}(\lambda_1)-\widetilde{G}(\lambda_2)$ 
can be obtained as a resultant of type \eqref{eq:alpha-beta}:
\begin{align*}
  p_{\Delta \widetilde{G}}(x)
  &= \Res(p_{\widetilde{G}}(x+y),p_{\widetilde{G}}(y),y)\\
  &= x^4   p_{\Delta \widetilde{G}}^{(1)}(x)\,    p_{\Delta \widetilde{G}}^{(2)}(x)
\end{align*}
where the nontrivial factors are
\begin{subequations}
  \label{eq:pDeltaGtilde12}
  \begin{align}
      \label{eq:pDeltaGtilde1}
p_{\Delta \widetilde{G}}^{(1)}(x)
&=
  \begin{multlined}[t]
x^4 -s^2 ( 2 z^3 s+9 z^2 s^2+6 z s^3+s^4 -8 z^2 -16 z s -4 s^2 ) x^2+s^4 z^4 (
 z^2 s^2 -8 z s -4 s^2+16 )
  \end{multlined}
    \\
      \label{eq:pDeltaGtilde2}    
p_{\Delta \widetilde{G}}^{(2)}(x)
&=
    \begin{multlined}[t]
      \makeatletter{}x^8 -2 s^2 \bigl( 3 z^3 s+9 z^2 s^2+6 z s^3+s^4 -4 z^2 -8 z s -2 s^2 \bigr)
x^6
\\
+s^4\bigl( 9 z^6 s^ 2+54 z^5 s^3+117 z^4 s^4+114 z^3 s^5+54 z^ 2 s^6+12 z
s^7+s^8 -16 z^5 s -110 z^4 s^2
\\
-202 z^3 s^3 -140 z^2 s^4 -40 z s^5 -4 s^6+16 z^
4+96 z^3 s+120 z^2 s^ 2+48 z s^3+6 s^4\bigr) x^4
\\
-s^7 \bigl( 4 z^9 s^ 2+9 z^8
s^3+6 z^7 s^ 4+z^6 s^5 -8 z^8 s+14 z^7 s^2+60 z^6 s^ 3+54 z^5 s^4+18 z^4 s^5
\\
+2z^3 s^6
+32 z^ 7+136 z^6 s+400 z^5 s^ 2+580 z^4 s^3+418 z^3 s^4+154 z^2 s^5+28 z
s^6
\\
+2 s^7 -128 z^5 -416 z^ 4 s -448 z^3 s^2 -216 z^ 2 s^3 -48 z s^4 -4 s^5\bigr) x^2
\\
+s^{10} \left( 4 z^5+5 z^4 s+z^3 s^2+16 z^ 3+20 z^2 s+8 z s^2+s^3
\right)^2
    \end{multlined}
  .
\end{align}
\end{subequations}

Next we compute an equation for the divided difference
$G_T(s)=\frac{\Delta\widetilde{G}(\lambda)}{s\Delta\lambda}$
as a resultant of type \eqref{eq:alpha/beta}.
We do this for the two branches
\eqref{eq:pDeltaGtilde12} separately.
\subsubsection{Branch  1 \eqref{eq:pDeltaGtilde1}}
In order to accomodate the factor $1/s$ we multiply $x$ with $s$ in
$p_{\Delta\widetilde{G}}$ and compute the resultant
\begin{equation*}
  \Res(p_{\Delta\widetilde{G}}^{(1)}(sxy), p_{\Delta\lambda}(y),y)
  =
        s^8\,p_{G_T}^{(1)}(x, s)^2
\end{equation*}
where the nontrivial factor
\begin{equation*}
  p_{G_T}^{(1)}(x, s)
  =
    \begin{multlined}[t]
    \makeatletter{}s^2 \bigl( s+4 z \bigr)^2 x^4 -s \bigl( s+4 z \bigr) \bigl( s^4+6 s^3 z+9 s^2
z^2+2 s z^3 -4 s^2 -16 s z -8 z^2 \bigr) x^2
\\
+z^4 \bigl( s z -2 s -4 \bigr)\bigl( s z+2 s -4 \bigr)
    \end{multlined}
\end{equation*}
is a branch of the equation for $G_T(s)$.
To see whether it is the correct branch we switch to
the moment generating function $\mgf_T(s) =\frac{1}{s} G_1(1/s)$.
It is a root of the nontrivial factor of the numerator of the function
\begin{equation*}
  p_{G_T}^{(1)}(sx,1/s) =
  \begin{multlined}[t]
  \makeatletter{}s^{-4}\bigl(
  s^4( 4  s  z+1 )^2  x^4
 +z^4  s^2( 16  s \sp 2 -8  s  z+z^2 -4 )
\\  
+( 32  s^5  z^3 -8  s^4  z^4+72  s^4  z^2 -38  s^3  z^3+32  s^3  z -33  s^2
   z^2+4  s^2 -10  s  z -1 ) x^2
\bigr) .    
  \end{multlined}
\end{equation*}
However at $s=0$ this means that $\mgf_T(0)$ satisfies the equation
$$
-x^2=0
$$
which is impossible since $\mgf_T(0)=1$.

\subsubsection{Branch 2 \eqref{eq:pDeltaGtilde2}}
As for the first branch we begin with the resultant
\begin{equation*}
    \Res(p_{\Delta\widetilde{G}}^{(2)}(sxy), p_{\Delta\lambda}(y),y)
    = s^{20}(s+4z)^4\, p_{G_T}^{(2,1)}(x,s)^2\, p_{G_T}^{(2,2)}(x,s)^2
\end{equation*}
where the nontrivial factors are
\begin{subequations}
  \begin{align}
    p_{G_T}^{(2,1)}(x, s) &= 
    \begin{multlined}[t]
    \makeatletter{}s \bigl( s+4 z \bigr) x^4 -2 s \bigl( s+z \bigr) \bigl( s+4 z \bigr) x^3
\\
+\bigl(
s^4+6 s^3 z+9 s^2 z^2+5 s z^3+2 s^2+8 s z+4 z^2 \bigr) x^2
\\
-\bigl( s+z \bigr)
\bigl( s z^3+2 s^2+8 s z+4 z^2 \bigr) x+s z^3+z^4+s^2+4 s z+4 z^2
  \end{multlined}
\label{eq:pGT21}
    \\
    p_{G_T}^{(2,2)}(x, s) &= 
    \begin{multlined}[t]
    \makeatletter{}s \bigl( s+4 z \bigr) x^4+2 s \bigl( s+z \bigr) \bigl( s+4 z \bigr) x^3
\\
+\bigl(
s^4+6 s^3 z+9 s^2 z^2+5 s z^3+2 s^2+8 s z+4 z^2 \bigr) x^2
\\
+\bigl( s+z \bigr)
\bigl( s z^3+2 s^2+8 s z+4 z^2 \bigr) x+s z^3+z^4+s^2+4 s z+4 z^2
 .
    \end{multlined}
\label{eq:pGT22}
  \end{align}
\end{subequations}

To see which subbranch yields the correct solution we switch to
the moment generating function $\mgf_T(s) =1/s G_1(1/s)$.
\begin{description}
 \item [Subbranch 2.1 \eqref{eq:pGT21}]
  The equation for the mgf
  is a root of the numerator of the polynomial
  \begin{equation}
    \label{eq:pMT21}
    p_{\mgf_T}^{(2,1)}(sx,1/s)
    =
    \begin{multlined}[t]
      \makeatletter{}    s^{-2}
    \Bigl(
s^4 \left( 4 s z+1 \right)x^4
-2 s^2 \left( s z+1 \right)\left( 4 s z+1 \right)x^3
\\
+\left( 4 s^4 z^2+5 s \sp 3 z^3+8 s^3 z+9 s^2 z^2+2 s^2+6 s z+1 \right)x^2
\\
-\left( s z+1 \right)\left( 4 s^2 z^2+s z^3+8 s z+2 \right)x
+s^2 z^4+4 s^2 z^2+s z^3+4 s z+1 
    \Bigr)
 .
    \end{multlined}
  \end{equation}
  At $s=0$ this implies that $\mgf_T(0)$ satisfies the equation
  $$
  x^2-2x+1=0
  ,
  $$
  i.e., $\mgf_T(0)=1$ and this is the correct branch.

 \item [Subbranch 2.2 \eqref{eq:pGT22}]
  The equation for the mgf
  is a root of the numerator of polynomial
  \begin{equation*}
    p_{\mgf_T}^{(2,2)}(sx,1/s) =
    \begin{multlined}[t]
      \makeatletter{}    s^{-2}
    \Bigl(
s^4 \left( 4 s z+1 \right)x^4
    +2 s^2 \left( s z+1 \right)\left( 4 s z+1    \right)x^3
\\   +\left( 4 s^4 z^2+5 s^3 z^3+8 s^3 z+9 s^2 z^2+2 s^2+6 s z+1
    \right)x^2
\\
    +\left( s z+1 \right)\left( 4 s^2 z^2+s z^3+8 s z+2 \right)x
    +s^2 z^4+4 s^2 z^2+s z^3+4 s z+1 
    \Bigr)
 .
    \end{multlined}
  \end{equation*}
  At $s=0$ this implies that $\mgf_T(0)$ satisfies the equation
  $$
  x^2+2x+1=0
  ,
  $$
  i.e., $\mgf_T(0)=-1$, which is not the correct solution.
\end{description}

We conclude that $G_T(s)$ is a root of  the polynomial
$p_{G_T}(x,s)=p_{G_T}^{(2,1)}(x,s)$ from \eqref{eq:pGT21}.

\subsection{An equation for the subordination function}
We use \eqref{eq:mainthm:deltaequation2}
to compute an algebraic equation for the subordination function $\delta(z)$.
First we obtain an equation for the Boolean cumulant generating function $\tilde{\eta}_T(s)$ by substituting
$\mgf_T(s)=\frac{1}{1-s\tilde{\eta}_T(s)}$ into equation \eqref{eq:pMT21}:
The numerator of the resulting rational expression is
\begin{equation*}
  p_{\mgf_T}(1/(1-sx),s)
  = \frac{s^2}{(1-sx)^4}
    \,
    p_{\tilde{\eta}_T}(x,s)
\end{equation*}
where
\begin{equation*}
  p_{\tilde{\eta}_T}(x,s) =
  \begin{multlined}[t]
    \makeatletter{}s^2 \left( s^2 z^4+4 s^2 z^2+s z^3+4 s z+1 \right)x^4
\\    
    +s \left( 4 s^3 z^3 -3 s^2 z^4 -4 s^2 z^2 -3 s z^3 -6 s z -2 \right)x^3
\\    
    +\left( 4 s^4 z^2 -7 s^3 z^3+8 s^3 z+3 s^2 z^4 -3 s^2 z^2+2 s^2+3 s z^3+1
    \right)x^2
\\    
    +\left( 2 s^2 z^3 -6 s^2 z -s z^4+2 s z^2 -2 s -z^3+2 z \right)x
\\    
    +4 s^3 z -4 s^2 z^2+s^2+s z^3 -2 s z+z^2 
  \end{multlined}
\end{equation*}
is an equation for $\tilde{\eta}_T(s)$.
The shifted Boolean cumulant generating function of the semicircle law
satisfies the equation \eqref{eq:etatildesemicircleimplicit},
or equivalently
$$
p_{\delta,\eta}(\delta(z),\tilde{\eta}_T(\delta(z)))=0
$$
where
$$
p_{\delta,\eta}(x,y) = (x^2+1)y-x
.
$$
We get an equation for $\delta(z)$ by eliminating $\tilde{\eta}$,
i.e., taking the resultant
\begin{equation}
  \label{eq:pdelta}
  \begin{aligned}
  p_\delta(y,z)
  &:=  \Res(p_{\mgf_T}(y, x), p_{\delta,\eta}(x,y), y)\\
  & =
  \begin{multlined}[t]
    \makeatletter{}4 z x^{11}
+x^{10}
+16 z x^9
+\left( -8 z^2+4 \right)x^8
+14 z x^7
+\left( -20 z^2+4 \right)x^6
\\
+\left( 5 z^3 -6 z \right)x^5
-7 z^2 x^4
+\left( 6 z^3 -4 z \right)x^3
+\left( -z^4+2 z^2 \right)x^2
+z^2 
  \end{multlined}
  \end{aligned}
\end{equation}

\subsection{An equation for the Cauchy transform of $X+XYX$}
Finally we are ready to evaluate the integral \eqref{eq:CauchyTrCantConv}
by the same method as in section~\eqref{ssec:CauchyT} 
\begin{equation}
  \label{eq:CauchyX+XYX}
G_{X+XYX}(z)=\int\frac{1}{z-t-\delta(z)t^2}\,d\mu(t)
= \frac{1}{\delta(z)}\frac{G_\mu(\lambda_1)-G_\mu(\lambda_2)}{\lambda_1-\lambda_2}
\end{equation}
where $\lambda_i$ satisfy the equation $p_\lambda(\lambda_i,\delta(z))=0$ where
$$
p_\lambda(x,y)=yx^2+x-z
.
$$
and $G_X(z)$ is the Cauchy transform of the standard semicircle law and
satisfies the Zhukovsky equation~\eqref{eq:WignerCauchy}.
Thus $G_\mu(\lambda)$ satisfies $p_G(G(\lambda),\delta(z),z)=0$ where $p_G$ is the resultant
\begin{equation*}
  \begin{aligned}
    p_G(x,y,z)
    &= \Res(yt^2+t-z, x^2-tx+1, t) \\
    &=yx^4+x^3+(2y-z)x^2+x+y
\end{aligned}
\end{equation*}
From this we compute the resultant  \eqref{eq:alpha-beta} and obtain
$$
\Res(p_G(x+x',y,z),p_G(x',y,z),x') = p_{\Delta G}^{(1)}(x,y,z)\, p_{\Delta G}^{(2)}(x,y,z)
$$
where
\begin{subequations}
\begin{align}
  p_{\Delta G}^{(1)}(x,y,z)
  &= y^2 x^4 -( 2 y z-8 y^2 +1)    x^2+z^2 -8 y z+16 y^2 -4
    \label{eq:pDeltaG1}
  \\
  p_{\Delta G}^{(2)}(x,y,z)
  &=
    \begin{multlined}[t]
      y^4 x^8
      -( 6 y^3 z-8 y^4 +2 y^2 )x^6
      +( 9 y^2 z^2+( -16 y^3+6 y )z+16 y^4 -6 y^2+1 )x^4
      \\
      -( 4 y z^3-( 8 y^2 -1 )z^2+( 32 y^3-10 y )z +8 y^2-2 )x^2
      +16 y^2 z^2+8 y z+1 
    \end{multlined}
    \label{eq:pDeltaG2}  
\end{align}
\end{subequations}
which yields two possible equations for $p_{\Delta G}(G(\lambda_1)-G(\lambda_2),\delta(z),z)=0$.

Similarly we compute an equation for $\Delta\lambda=\lambda_1-\lambda_2$.
First apply the resultant \eqref{eq:alpha-beta} to $p_\lambda$:
$$
\Res(p_\lambda(x+x',y),p_\lambda(x',y),x') = x^2y^2 p_{\Delta\lambda}(x,y)
$$
then the nontrivial factor
$$
p_{\Delta\lambda}(x,y) =
y^2x^2-4yz-1
$$
yields the equation
$p_{\Delta\lambda}(\lambda_1-\lambda_2,\delta(z)) =0$.
Finally the resultant \eqref{eq:alpha/beta} 
$$
\Res(p_{\Delta G}(xx'y,y,z),p_\lambda(x',y))
$$
yields an equation for  \eqref{eq:CauchyX+XYX}.
We consider the two branches separately.
\subsubsection{Branch 1 \eqref{eq:pDeltaG1}}
The resultant is
$$
\Res(p_{\Delta G}^{(1)}(xx'y,y,z),p_\lambda(x',y))
=y^8\, p_{X+XYX,\delta}^{(1)}(x,y,z)^2
$$
where
\begin{equation*}
  p_{X+XYX,\delta}^{(1)}(x,y,z) = y^2(4yz+1)^2x^4-(2yz-8y^2+1)(4yz+1)x^2+(z-4y-2)(z-4y+2)
\end{equation*}
is a candidate for the equation $p_{X+XYX,\delta}^{(1)}(G_{X+XYX}(z),\delta(z),z) = 0$.
To check whether it is the correct branch we substitute $\delta=0$ in the integral
\eqref{eq:CauchyX+XYX} which results in the Cauchy distribution of the semicircle law.
On the other hand, substituting $y=0$ in the equation yields
$$
p_{X+XYX,\delta}^{(1)}(x,0,z) = z^2-x^2-4
$$
which is not the equation of the semicircle law.

\subsubsection{Branch 2 \eqref{eq:pDeltaG2}}
The resultant is
$$
\Res(p_{\Delta G}^{(2)}(xx'y,y,z),p_\lambda(x',y))
=y^{16}(4yz+1)^4 p_{X+XYX,\delta}^{(2,1)}(x,y,z)^2 p_{X+XYX,\delta}^{(2,2)}(x,y,z)^2
$$
where
\begin{subequations}
  \begin{align}
    \label{eq:pxxyxdelta21}
  p_{X+XYX,\delta}^{(2,1)}(x,y,z)
  &= y^2(4yz+1)x^4-2y(4yz+1)x^3+(5yz+4y^2+1)x^2-(z+4y)x+1
    \\
  p_{X+XYX,\delta}^{(2,2)}(x,y,z)
  &= y^2(4yz+1)x^4+2y(4yz+1)x^3+(5yz+4y^2+1)x^2+(z+4y)x+1
\end{align}
\end{subequations}
To select the correct subbranch we  substitute $y=0$ and obtain
\begin{align*}
  p_{X+XYX,\delta}^{(2,1)}(x,0,z) &= x^2-xz+1\\
  p_{X+XYX,\delta}^{(2,2)}(x,0,z) &= x^2+xz+1
\end{align*}
and since  $p_{X+XYX,\delta}^{(2,1)}(x,y,z)$ is the only branch
which makes the semicircle vanish at $y=0$ we conclude that it is the correct one.
\subsection{Final Elimination Step}
Finally we eliminate $\delta(z)$ from equations \eqref{eq:pxxyxdelta21} and
\eqref{eq:pdelta} and we obtain an equation
$$
\Res(p_{X+XYX,\delta}^{(2,1)}(x,y,z),p_\delta(y,z),y)
=16z\, p_{G_{X+XYX}}^{(1)}(x,z)\, p_{G_{X+XYX}}^{(2)}(x,z)\, p_{G_{X+XYX}}^{(3)}(x,z),
$$
with three branches
\begin{subequations}
\begin{align}
  p_{G_{X+XYX}}^{(1)}(x,z)
  &=
    \begin{multlined}[t]
      \makeatletter{}16z^3\,x^{11}+(16z^4+16z^2+1)\,x^{10}-z(32z^2-5)\,x^9
\\
+3(4z^2+3)\,x^8-8z(3z^2+2)\,x^7+2(28z^2+11)\,x^6
\\
-52z\,x^5 +(9z^2+22)\,x^4-16z\,x^3+9 \,x^2-z\,x+1
    \end{multlined}
\label{eq:pGXXYXcorrect}
    \\
  p_{G_{X+XYX}}^{(2)}(x,z)
  &=
    \begin{multlined}[t]
      \makeatletter{}16 z^5 \,x^{11}-( 16 z^6 +16 z^4 +z^2 ) \,x^{10}+( 96 z^5+9 z^3 ) \,x^9
-( 192 z^4+13 z^2 ) \,x^8
\\
+( 8 z^5+176 z^3+4 z ) \,x^7
-( 40 z^4 +74 z^2 ) \,x^6
+( 85 z^3+14 z ) \,x^5
\\
-( z^4 +93 z^2 +4 ) \,x^4
+( 4 z^3+52 z ) \,x^3
-( 6 z^2 +12 ) \,x^2
+4 z \,x -1
    \end{multlined}
\end{align}
\begin{multline*}
  p_{G_{X+XYX}}^{(3)}(x,z)
  = 
    \begin{multlined}[t]
      \makeatletter{}z^2 (960 z^6+32 z^4+31 z^2+1) \,x^{22}
\\
-z^2 ( 448 z^8-1504 z^6+75 z^4-409 z^2-22)\,x^{20}
\\
-2 z (64 z^{10}-800 z^8+1673 z^6+759 z^4+102 z^2+2) \,x^{19}
\\
+z^2 (64 z^{10}+480 z^8-443 z^6+6025 z\sp4+2317 z^2+201) \,x^{18}
\\
-z (384 z^{10}-1288 z^8+3682 z^6+7397 z^4+1575 z^2+50) \,x^{17}
\\
+(1300 z^{10}-7974 z\sp8+10148 z^6+7709 z^4+1205 z^2+4) \,x^{16}
\\
-2z (136 z^{10}+1712 z^8-7273 z^6+7538 z^4+3134 z^2+115 ) \,x^{15}
\\
+2 (936 z^{10}+2339 z^8-5897 z^6+7663 z^4+1649 z^2+24) \,x^{14}
\\
-z (6277 z^8-1267 z^6 -3132 z^4+10298 z^2+680) \,x^{13}
\\
+(321 z^{10}+13836 z \sp8-13869 z^6+2402 z^4+4377 z^2+197) \,x^{12}
\\
-z (2436 z^8+21743 z^6-22041 z^4+2631 z^2+1131) \,x^{11}
\\
+(8279 z^8+23938 z^6-19138 z^4+1284 z^2+361) \,x^{10}
\\
-z (68 z^8+16520 z\sp6+16753 z^4-10691 z^2+622) \,x^9
\\
+(468 z^8+21132 z^6+5447 z^4-3729 z^2+322) \,x^8
\\
-z (1407 z^6+17559 z\sp4-1244 z^2-470) \,x^7
\\
+(4 z^8+2383 z^6+9001 z^4-1718 z^2+134) \,x^6
\\
-6 z (4 z^6+407 z^4+393 z^2 -70) \,x^5
+(60 z^6+1502 z^4+61 z^2+21) \,x^4
\\
-z (80 z^4+503 z^2-79) \,x^3
+(60 z^4+63 z^2+1) \,x^2
-4 z (6 z^2-1) \,x+4 z^2
    \end{multlined}
\end{multline*}
\end{subequations}

Finally we use Newton polygons as described in Section~\ref{ssec:puiseux}
to determine the correct branch.
To this end we perform a shift to transform these equations into an
equation for the moment generating function $\mgfpsi(z) = \mgf(z)-1$.
This results in a polynomial with 210 terms,
of which 
we only reproduce here the critical ones which are required 
to draw the Newton polygons, i.e., for each power of $x$
we record the monomial with smallest degree in $z$:

\begin{subequations}
\begin{align*}
  \tilde{p}_{\mgfpsi_{X+XYX}}^{(1)}(x,z)
  &=
    \begin{multlined}[t]
      \makeatletter{}16 z^8 x^{11}
+16 z^6 x^{10}
+128 z^6 x^9
+444 z^6 x^8
-24 z^4 x^7
\\
-112 z^4 x^6
-220 z^4 x^5
+9z^2 x^4
+20 z^2 x^3
+15z^2 x^2
- x
+2z^2
    \end{multlined}
    \\
  \tilde{p}_{\mgfpsi_{X+XYX}}^{(2)}(x,z)
  &=
    \begin{multlined}[t]
      \makeatletter{}16 z^6 x^{11}
-16z^4 x^{10}
-64z^4 x^9
-48z^4 x^8
+8 z^2 x^7
\\
+16 z^2 x^6
+13z^2 x^5
- x^4
+10 z^2 x^3
+4z^2 x^2
+z^2 x
-z^8
    \end{multlined}
    \\
\tilde{p}_{\mgfpsi_{X+XYX}}^{(3)}(x,z)
  &= 
    \begin{multlined}[t]
      \makeatletter{}960z^{16} x^{22}
+21120z^{16} x^{21}
-448z^{12} x^{20}
-128 z^{10}x^{19}
+64z^8 x^{18}
\\
+768z^8 x^{17}
+4564z^8 x^{16}
-272 z^6 x^{15}
-2208 z^6 x^{14}
-8629z^6 x^{13}
\\
+321z^4 x^{12}
+1416z^4 x^{11}
+2669z^4 x^{10}
-68z^2 x^9
-144z^2 x^8
-111z^2 x^7
\\
+4 x^6
-51z^2 x^5
-16z^2 x^4
-4z^2 x^3
+37z^4 x^2
+8z^4 x
+z^4 
    \end{multlined}
\end{align*}
\end{subequations}
The corresponding Newton polygons are shown in Fig.~\ref{fig:newton}.
\begin{figure}[h]
    \begin{subfigure}[t]{0.5\textwidth}
    \centering
\begin{tikzpicture}[scale=0.5]
  \draw[step=1] (0,0) grid (11,8);
  \draw[fill] (11,8) circle (2pt);
  \draw[fill] (10,6) circle (2pt);
  \draw[fill] (9,6) circle (2pt);
  \draw[fill] (8,6) circle (2pt);
  \draw[fill] (7,4) circle (2pt);
  \draw[fill] (6,4) circle (2pt);
  \draw[fill] (5,4) circle (2pt);
  \draw[fill] (4,2) circle (2pt);
  \draw[fill] (3,2) circle (2pt);
  \draw[fill] (2,2) circle (2pt);
  \draw[fill] (1,0) circle (2pt);
  \draw[fill] (0,2) circle (2pt);
  \draw (0,2)--(1,0)--(10,6)--(11,8);
\end{tikzpicture}
        \caption{$\tilde{p}^{(1)}(x,z)$}
    \end{subfigure}    ~ 
    \begin{subfigure}[t]{0.5\textwidth}
    \centering
\begin{tikzpicture}[scale=0.5]
  \draw[step=1] (0,0) grid (11,8);
  \draw[fill] (11,6) circle (2pt);
  \draw[fill] (10,4) circle (2pt);
  \draw[fill] (9,4) circle (2pt);
  \draw[fill] (8,4) circle (2pt);
  \draw[fill] (7,2) circle (2pt);
  \draw[fill] (6,2) circle (2pt);
  \draw[fill] (5,2) circle (2pt);
  \draw[fill] (4,0) circle (2pt);
  \draw[fill] (3,2) circle (2pt);
  \draw[fill] (2,2) circle (2pt);
  \draw[fill] (1,2) circle (2pt);
  \draw[fill] (0,8) circle (2pt);
  \draw (0,8)--(1,2)--(4,0)--(10,4)--(11,6);
\end{tikzpicture}
        \caption{$\tilde{p}^{(2)}(x,z)$}
    \end{subfigure}

\bigskip{}

    \begin{subfigure}[t]{0.9\textwidth}
    \centering
\begin{tikzpicture}[scale=0.5]
  \draw[step=1] (0,0) grid (22,16);
  \draw[fill] (22,16) circle (2pt);
  \draw[fill] (21,16) circle (2pt);
  \draw[fill] (20,12) circle (2pt);
  \draw[fill] (19,10) circle (2pt);
  \draw[fill] (18,8) circle (2pt);
  \draw[fill] (17,8) circle (2pt);
  \draw[fill] (16,8) circle (2pt);
  \draw[fill] (15,6) circle (2pt);
  \draw[fill] (14,6) circle (2pt);
  \draw[fill] (13,6) circle (2pt);
  \draw[fill] (12,4) circle (2pt);
  \draw[fill] (11,4) circle (2pt);
  \draw[fill] (10,4) circle (2pt);
  \draw[fill] (9,2) circle (2pt);
  \draw[fill] (8,2) circle (2pt);
  \draw[fill] (7,2) circle (2pt);
  \draw[fill] (6,0) circle (2pt);
  \draw[fill] (5,2) circle (2pt);
  \draw[fill] (4,2) circle (2pt);
  \draw[fill] (3,2) circle (2pt);
  \draw[fill] (2,4) circle (2pt);
  \draw[fill] (1,4) circle (2pt);
  \draw[fill] (0,4) circle (2pt);
  \draw (0,4)--(3,2)--(6,0)--(18,8)--(22,16);
\end{tikzpicture}
        \caption{$\tilde{p}^{(3)}(x,z)$}
    \end{subfigure}

    \caption{Newton polygons of the three factors}
\label{fig:newton}
\end{figure}
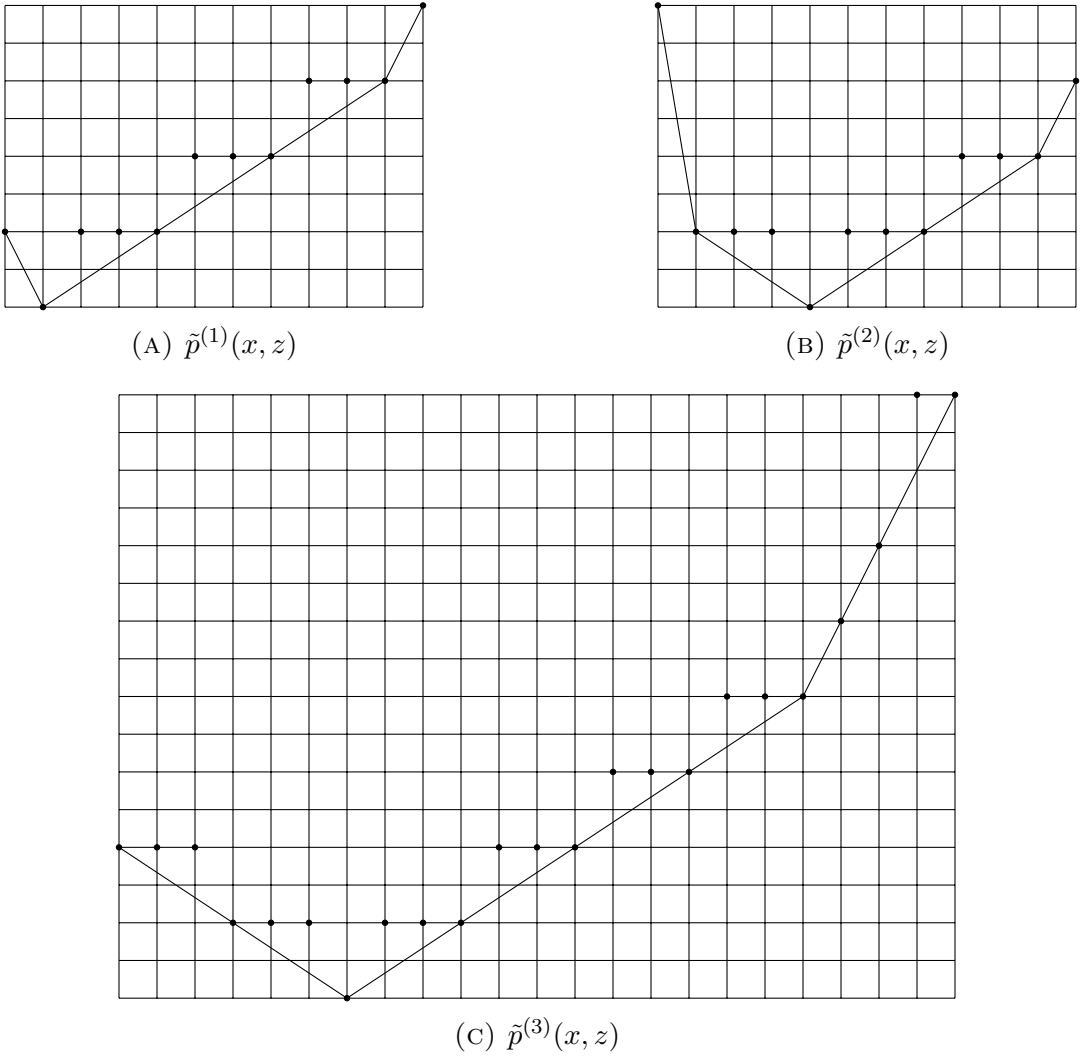

Thus for the first factor we find the solutions $\gamma\in\{-2,-2/3,2\}$, namely
\begin{subequations}
\begin{align}
  \mgfpsi(z) &= -z^{-2}(1+\Ord(z)) & &&\text{$1$ solution}\\
  \mgfpsi(z) &= z^{-2/3}(c+\Ord(z)), & 16c^9-24c^6+9c^3-1&=0& \text{$9$ solutions}\\
  \mgfpsi(z) &= z^{2}(2+\Ord(z)) & &&\text{$1$ solution}
\label{eq:correctbranch}
\end{align}
\end{subequations}

For the second factor we find the solutions $\gamma\in\{-2,-2/3,2/3,6\}$, namely
\begin{subequations}
\begin{align*}
  \mgfpsi(z) &= z^{-2}(1+\Ord(z)) & &&\text{$1$ solution}\\
  \mgfpsi(z) &= z^{-2/3}(c+\Ord(z)), & 16c^6-8c^3+1&=0& \text{$6$ solutions}\\
  \mgfpsi(z) &= z^{2/3}(c+\Ord(z)), & c^3-1&=0& \text{$3$ solutions}\\
  \mgfpsi(z) &= z^{6}(1+\Ord(z)) & &&\text{$1$ solution}
\end{align*}
\end{subequations}

For the third factor we find the solutions $\gamma\in\{-2,-2/3,2/3\}$, namely
\begin{subequations}
\begin{align*}
  \mgfpsi(z) &= z^{-2}(c+\Ord(z)), & 120c^4-61c^2-16c+8&=0& \text{$4$ solutions}\\
  \mgfpsi(z) &= z^{-2/3}(c+\Ord(z)), & 64c^{12}-272c^9+321c^6-68c^3+4&=0& \text{$12$ solutions}\\
  \mgfpsi(z) &= z^{2/3}(c+\Ord(z)), & 4c^6-4c^3+1&=0& \text{$6$ solutions}
\end{align*}
\end{subequations}

We conclude that the first factor \eqref{eq:pGXXYXcorrect} is the correct one
and the solution is \eqref{eq:correctbranch}. 
We can now proceed to plot the density and compute some moments.

\subsection{Plot and spectral radius}
A picture of the imaginary part of the branches of this function is shown in
Fig.~\ref{fig:X+XYX:semi:density}.
The density is an algebraic function and satisfies an equation obtained as follows:
\begin{enumerate}[1.]
 \item substitute $x=u+iv$ and $z=t$ into equation \eqref{eq:pGXXYXcorrect} 
 \item separate real and imaginary part
 \item compute the resultant of these two polynomials with respect to $u$ 
 \item this results in an equation of degree $110$.
\end{enumerate}
A numerical picture of this ``twin peaks'' density is shown in 
Figure~\ref{fig:X+XYX:semi:density};
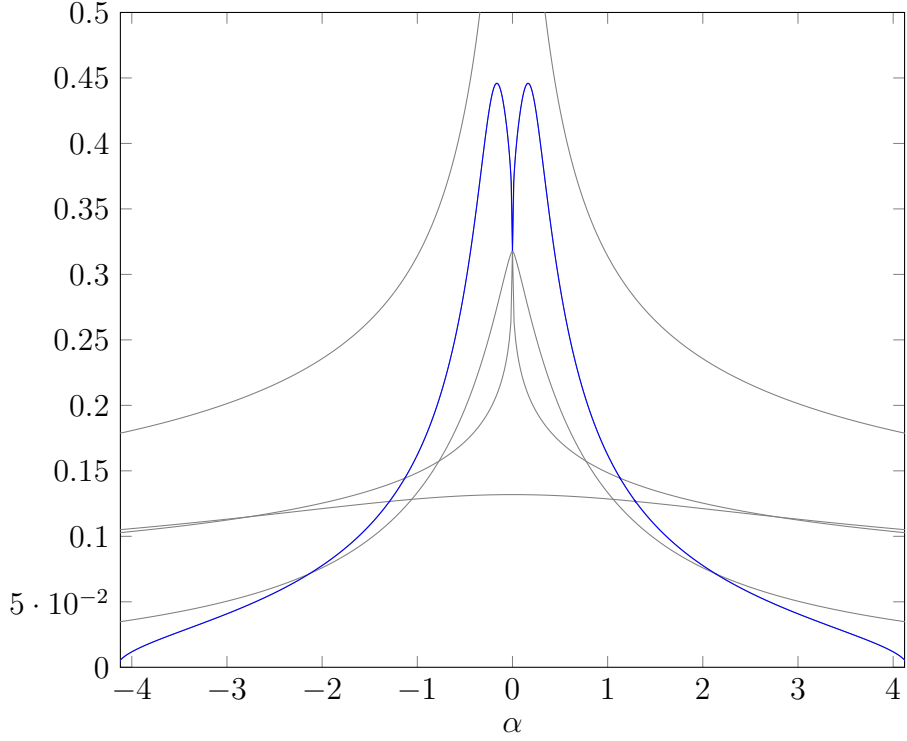
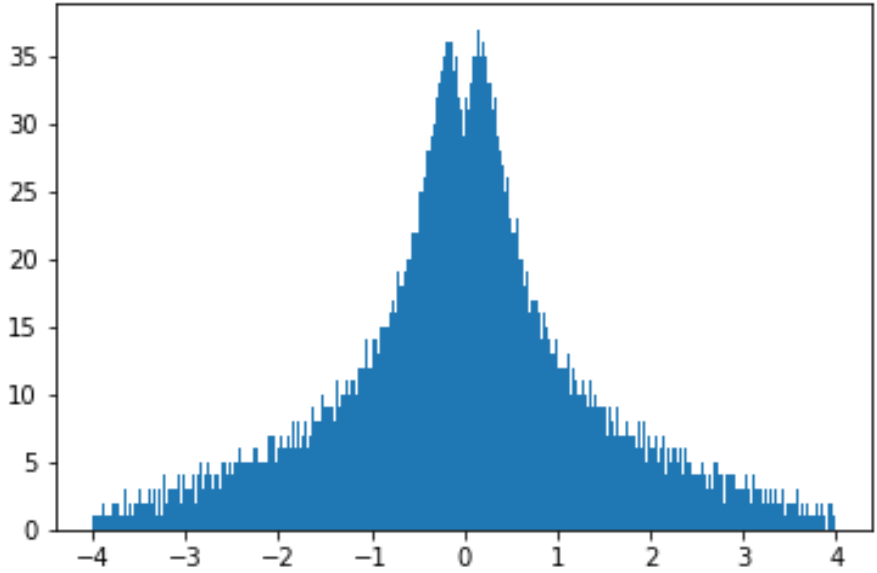
\begin{figure}[h]
  \centering{}
\begin{subfigure}{1\textwidth}
	\centering{}
\begin{tikzpicture}
 \begin{axis}[xlabel={$\alpha$},
  width=0.7\textwidth,
 xmin=-4.12,
 xmax=4.12,
 ymin=0,
 ymax=0.5
 ]
 \addplot[color=gray] table {pictures/semi-densityXYX500_1.dat};
 \addplot[color=gray] table {pictures/semi-densityXYX500_2.dat}; 
 \addplot[color=gray] table {pictures/semi-densityXYX500_3.dat}; 
 \addplot[color=gray] table {pictures/semi-densityXYX500_4.dat}; 
 \addplot[color=gray] table {pictures/semi-densityXYX500_5.dat}; 
 \addplot[color=blue] table {pictures/semi-densityXYX500_0.dat}; 
 \end{axis}
\end{tikzpicture}
  \caption{Plot of the density $\mu_{X+XYX}$ for Wigner law}
  \label{fig:X+XYX:semi:density}
\end{subfigure}

  \bigskip{}

\begin{subfigure}{1\textwidth}
	\centering{}
	\includegraphics[width=0.7\textwidth]{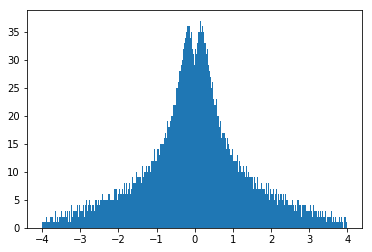}
	\caption{Histogram of the spectrum of a $4000\times4000$ random matrix model
		for $X+XYX$ for Wigner law}
	\label{fig:X+XYX:randmat}
\end{subfigure}
\caption{The spectral density of $X+XYX$ for Wigner law}
\end{figure}
for comparison, Figure~\ref{fig:X+XYX:randmat} shows the histogram of the spectrum of a sample of a $4000\times 4000$ random matrix of the form
$X+XYX$ where $X$ and $Y$ are complex standard normal Wigner matrices.

The spectral radius is a real root of the polynomial
\begin{multline*}
40310784 \, t^{18} -2717805312 \, t^{16} -41420635929 \, t^{14}+402631122484 \,
t^{12}+10964310641760 \, t^{10}
\\
+74046100039296 \, t^8+169907747390208 \,
t^6+3120150062080 \, t^4+544693026816 \, t^2+12230590464
\end{multline*}
which are approximately
$$
\pm{8.848639498045666}, \pm{4.156072921386361}
.
$$

\subsection{Recurrence}
It is well known that algebraic functions are holonomic,
i.e., they satisfy a homogeneous differential equation with polynomial coefficients
\cite{KauersPaule:2011:tetrahedron}.
Comparing coefficients of the differential equation one can then extract
a recurrence relation with polynomial coefficients and thus the sequence
of moments is holonomic.

The existence of the differential equation
follows from the observation that the algebra
over the rational functions $\IC(z)$ generated by an algebraic
function is finite dimensional and closed under differentiation and therefore
there is a nontrivial vanishing finite linear combination of the derivatives.
Thus one way to obtain such a differential equation 
is Cockle's algorithm \cite{Cockle:1861} which consists in the construction
of a minimal linear dependent set of derivatives.
The obtained equation is essentially unique and  called the \emph{differential resolvent};
for better algorithms see \cite {BostanChyzakSalvyLecerf:2007}.
In our case Cockle's algorithm produces a differential equation of order 11
with coefficients of degree up to 148 which is too large to reproduce here.
The corresponding recurrence for the moments allows their fast and efficient computation
and yields the sequence
$$
0,2,0,14,0,138,0,1586,0,19891,0,263948,0,3643590,0,51786474,0,752757867,\dots
$$

\subsection{Arcsine law}
Let $u$ and $v$  be free Haar unitaries, e.g.,
arising from the left regular representation of the free group.
Then $u+u^*$ and $v+v^*$ are distributed according to the arcsine law
$$
d\nu(t) = \frac{dt}{\pi\sqrt{4-t^2}}
$$
and the Cauchy transform satisfies the quadratic equation
$G_\nu(z)^2(z^2-4)-1=0$.
Redoing the preceding calculation with this equation instead of 
\eqref{eq:WignerCauchy}
and the relation
$$
\frac{2\widetilde{\eta}_\nu(s)}{4+\widetilde{\eta}_\nu(s)^2}=s
$$
instead of \eqref{eq:etatildesemicircleimplicit}      
one obtains after some elimination steps the final equation
\begin{multline*}
( z^2 -100 )^{2} ( z^2 -36 )^{2} {( z^2 -4 )}  ( 64z^2 -1 ) z x^{{11}}
\\
+{( z^2 -{100} )} {( z^2 -36 )}   {( {{320} {{z}^{8}}}
    -{{11011} {{z}^{6}}} -{{517228} {{z}^{4}}} -{{905232} {{z}^{2}}}
    -{14400} )} {{x}^{{10}}}
\\
+{{( {{576} \, {{z}^{{11}}}} -{{34175} \, {{z} 
\sp {9}}} -{{1963536} \, {{z}^{7}}}+{{29614176} \, {{z}^
{5}}}+{{3314048768} \, {{z}^{3}}} -{{34963200} \, z} 
)}
\, {{x}^{9}}}
\\
+{{( {{320} \, {{z}^{{10}}}}+{{48523} \, {{z}^
{8}}} -{{1787504} \, {{z}^{6}}} -{{136386272} \, {{z}^{4}}} 
-{{917888768} \, {{z}^{2}}} -{112723200} 
)}
\, {{x}^{8}}}
\\
-( 320 \, z^9-59014 \, z^7-2962872 \, z^5 +42200800 \, z^3 +43450752 \, z )
\, {{x}^{7}}
\\
-
( 576 \, z^8 +7822 \, z^6-1641032 \, z^4-62892128 \, z^2 +72966784 )
\, x^6
\\
-( 320 \, z^7 +29838 \, z^ 5 +2132368 \,      z^3-26113824 \, z ) \, x^5
\\
-( 64 \, z^6 +7034 \, z^ 4 +1647088 \,      z^2 +12050592 ) \, x^4
\\
+{{( {{2187} \, {{z}^{3}}} -{{220972} \, z} 
)}
\, {{x}^{3}}}+{{( {{385} \, {{z}^{2}}} -{11028} 
)}
\, {{x}^{2}}} -{3 \, z \, x} -1 =0
\end{multline*}
for the Cauchy transform $G_{X+XYX}(z)$.
Pictures of the onion shaped density and a random matrix approximation
are shown in Figure~\ref{fig:X+XYX:asin}.

\begin{figure}[h]
  \centering{}
\begin{subfigure}{1\textwidth}
	\centering{}
\begin{tikzpicture}
 \begin{axis}[xlabel={$\alpha$},
  width=0.7\textwidth,
 xmin=-7.25,
 xmax=7.25,
 ymin=0,
 ymax=0.35
 ]
 \addplot[color=gray] table {pictures/asin-densityXYX500_1.dat};
 \addplot[color=gray] table {pictures/asin-densityXYX500_2.dat}; 
 \addplot[color=gray] table {pictures/asin-densityXYX500_3.dat}; 
 \addplot[color=gray] table {pictures/asin-densityXYX500_4.dat}; 
 \addplot[thick,color=blue] table {pictures/asin-densityXYX500_5.dat}; 
 \end{axis}
\end{tikzpicture}
  \caption{Plot of the density $\mu_{X+XYX}$ for arcsine law}
  \label{fig:X+XYX:asin:density}
\end{subfigure}

\begin{subfigure}{1\textwidth}
	\centering{}
	\includegraphics[width=0.7\textwidth]{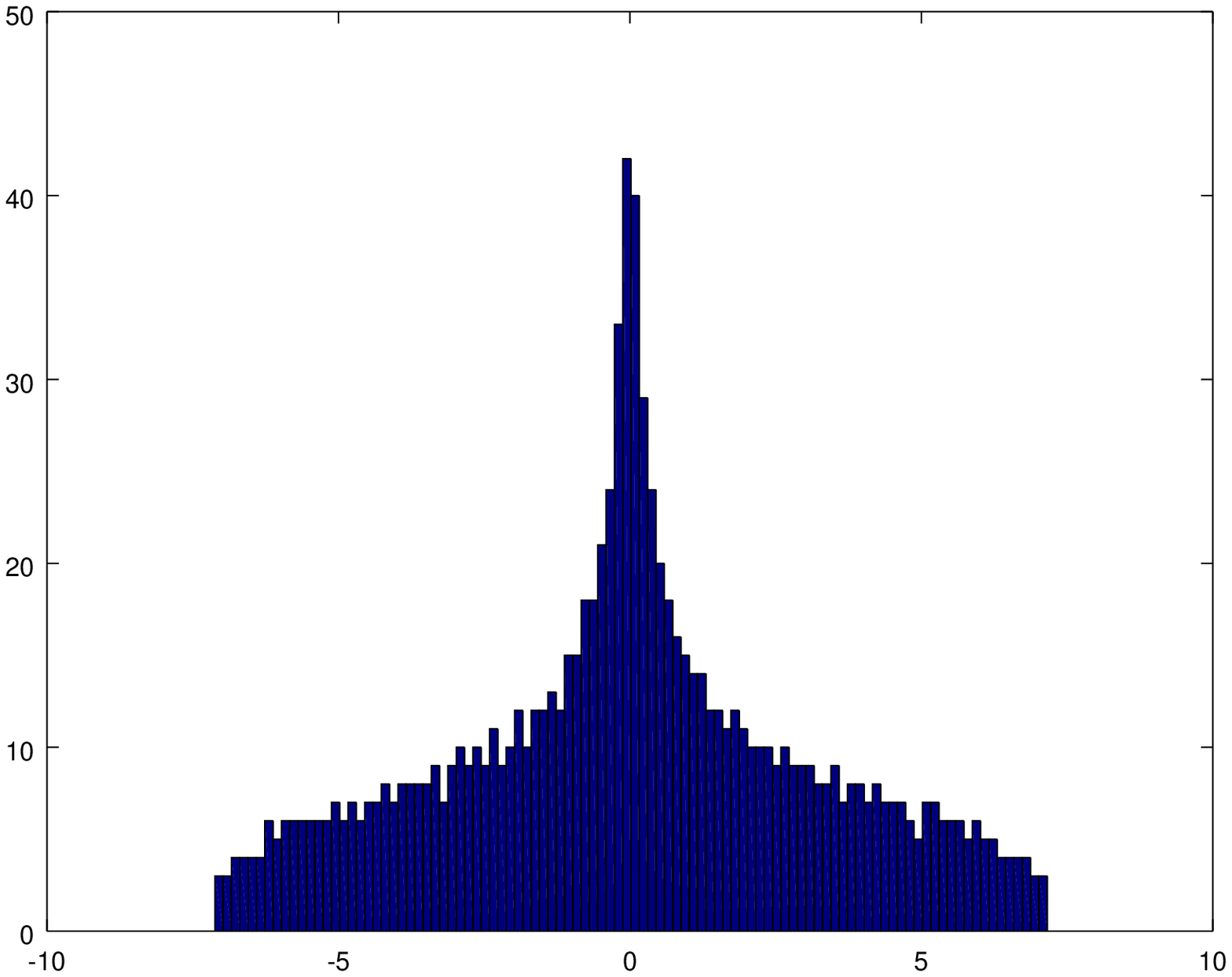}
	\caption{Histogram of the spectrum of a $1000\times1000$ random matrix model
		for $X+XYX$ for arcsine law}
	\label{fig:X+XYX:asin:randmat}
\end{subfigure}
\caption{The spectral density of $X+XYX$ for arcsine law}
	\label{fig:X+XYX:asin}
\end{figure}

\bibliographystyle{abbrv}

\end{document}